\documentclass{amsart}

\usepackage{amssymb,amsfonts, amsmath, amsthm}
\usepackage[all,arc]{xy}
\usepackage{enumerate}
\usepackage{mathrsfs}
\usepackage{mathtools}
\usepackage[mathscr]{euscript}
\usepackage{array}
\usepackage{tikz}
\usepackage{tikz-cd}
\usepackage{textcomp}
\usetikzlibrary{decorations.pathmorphing,shapes}
\usepackage[pagebackref, colorlinks=true, linkcolor=blue, citecolor = red]{hyperref}
 \renewcommand*{\backref}[1]{}
\renewcommand*{\backrefalt}[4]{({%
		\ifcase #1 Not cited.%
		\or On p.~#2%
		\else On pp.~#2%
		\fi%
	})}

\usepackage[nameinlink,capitalise,noabbrev]{cleveref}
\crefname{subsection}{Subsection}{Subsection}
\crefname{equation}{Diagram}{Diagram}

\usetikzlibrary{patterns}

\usepackage{geometry}
\usepackage{todonotes}

\newcommand{\C}{\mathscr{C}}

\newcommand{\E}{\mathscr{E}}

\renewcommand{\O}{\mathscr{O}}

\newcommand{\U}{\mathscr{U}}

\newcommand{\Partial}{\mathscr{P}\text{artial}}
\newcommand{\Total}{\mathscr{T}\text{otal}}

\newcommand{\DA}{\mathscr{D}\mathscr{A}}

\newcommand{\Map}{\mathrm{Map}}
\newcommand{\Hom}{\mathrm{Hom}}
\newcommand{\colim}{\mathrm{colim}}
\newcommand{\Eq}{\mathrm{Eq}}
\newcommand{\Arr}{\mathscr{A}\mathrm{rr}}
\newcommand{\Coeq}{\mathrm{Coeq}}
\newcommand{\Cone}{\mathrm{Cone}}
\newcommand{\Sub}{\mathrm{Sub}}
\newcommand{\Ind}{\mathrm{Ind}}
\newcommand{\Part}{\mathscr{P}\mathrm{art}}
\newcommand{\uPart}{\underline{\Part}}

\newcommand{\blueH}{\textcolor{blue}{H}}
\newcommand{\bluepH}{\textcolor{blue}{pH}}

\newcommand{\comma}{,}
\newcommand{\leb}{[}
\newcommand{\reb}{]}

\newcommand{\dumbmacro}{ [X \times \mathbb{N} \to \mathbb{N}]^{[\mathbb{N}_1 \to \mathbb{N}]}}
\newcommand{\ds}{\displaystyle}

\newcommand{\comsq}[8]{
  \begin{tikzcd}[row sep=0.3in, column sep=0.3in]
    #1 \arrow[r, "#5"] \arrow[d, "#6"']
    \pgfmatrixnextcell #2 \arrow[d, "#7"] \\
    #3 \arrow[r, "#8"]
    \pgfmatrixnextcell #4
  \end{tikzcd}
}

\newcommand{\pbsq}[8]{
  \begin{tikzcd}[row sep=0.3in, column sep=0.3in]
    #1 \arrow[r, "#5"] \arrow[d, "#6"'] \arrow[dr, phantom, "\ulcorner", very near start]
    \pgfmatrixnextcell #2 \arrow[d, "#7"] \\
    #3 \arrow[r, "#8"']
    \pgfmatrixnextcell #4
  \end{tikzcd}
}

\newcommand{\adjun}[4]{\begin{tikzcd}[row sep/normal=0.5in,column sep/normal=0.5in]
 #1  \arrow[r, shift left=1.4, "#3"] \pgfmatrixnextcell
 #2 \arrow[l, shift left=1, "#4"] 
\end{tikzcd}
}

\newtheorem{theone}[equation]{Theorem}

\newtheorem{lemone}[equation]{Lemma}
\newtheorem{propone}[equation]{Proposition}
\newtheorem{corone}[equation]{Corollary}

\theoremstyle{definition}
\newtheorem{defone}[equation]{Definition}
\newtheorem{exone}[equation]{Example}

\theoremstyle{remark}
\newtheorem{remone}[equation]{Remark}

\numberwithin{equation}{subsection}

\title{Every Elementary Higher Topos has a Natural Number Object}

\author{Nima Rasekh}

\date{March 2021}
\keywords{elementary topos theory, higher category theory, natural number objects}
\subjclass[2020]{03G30, 18B25, 18N60, 55U35}

\address{{\'E}cole Polytechnique F{\'e}d{\'e}rale de Lausanne, SV BMI UPHESS, Station 8, CH-1015 Lausanne, Switzerland}

\email{nima.rasekh@epfl.ch}

\begin{document}

\begin{abstract}

We prove that every elementary $(\infty,1)$-topos has a natural number object. 
We achieve this by defining the loop space of the circle and showing that we can construct a natural number object out of it.
Part of the proof involves showing that various definitions of natural number objects (Lawvere, Freyd and Peano) agree with each other in 
an elementary $(\infty,1)$-topos.
As part of this effort we also study the internal object of contractibility in $(\infty,1)$-categories, which is of independent interest. 
Finally, we discuss various applications of natural number objects. In particular, we use it to define internal sequential 
colimits in an elementary $(\infty,1)$-topos.

\end{abstract}

\maketitle
\addtocontents{toc}{\protect\setcounter{tocdepth}{1}}

 \section{Introduction}\label{Introduction}
  
 \subsection{History $\&$ Motivation}\label{Subsec Motivation}
 One of the first results students learn in a standard algebraic topology course is that $\pi_1(S^1) = \mathbb{Z}$ \cite{hatcher2002algebraictopology}.
 As a matter of fact we can compute the {\it loop space} and show $\ds \Omega S^1 = \mathbb{Z}$. 
 From a categorical perspective this result should be surprising. 
 Recall that a loop space is the following pullback
 \begin{center}
  \pbsq{\Omega S^1}{*}{*}{S^1}{}{}{}{}.
 \end{center}
 Thus the result is saying that the finite limit of a finite CW-complex ($S^1$ has two cells) is an infinite CW-complex.
 Such a thing would never happen in sets. The finite limit of finite sets is always finite.
 The implication is that the higher homotopical structure in finite spaces is implicitly infinite
 and the loop space construction makes that explicit. The seemingly innocuous result has wide ranging logical consequences. 
 \par 
 In the past century various foundations of mathematics have been established: {\it set theory} \cite{fraenkelbarhillel1958settheory}, {\it type theory} \cite{church1940typetheory}, and {\it elementary topos theory} \cite{tierney1973elementarytopos}.
 In particular, using elementary toposes we can develop many results which have been developed classically using sets. 
 One key aspect is the {\it natural number object} \cite{lawvere1963algebraictheory}, which corresponds to the {\it axiom of infinity} in set theory.
 Using natural number objects we can construct {\it free finitary algebras} (such as free monoids) in an elementary topos and discuss {\it geometric theories} \cite{johnstone2002elephanti,johnstone2002elephantsii}.
 We even can construct elementary toposes that have {\it non-standard natural number objects} (for example via the {\it filter construction} \cite{adelmanjohnstone1982serreclasses}). 
 \par 
 There are now several methods to develop foundations in a homotopical setting. One approach, known as {\it homotopy type theory} or {\it univalent foundations} \cite{hottbook2013}, has already been used to prove many classical results from homotopy theory. 
 On the other side, an alternative approach via higher categories, know as {\it elementary $(\infty,1)$-topos theory} \cite{rasekh2018elementarytopos},
 is still in its early stages. 
 \par
 The goal of this paper is to prove that, unlike in the $1$-categorical setting, in the higher categorical setting the existence of a natural number object follows from the axioms of an elementary $(\infty,1)$-topos (and in fact even weaker conditions given in \cref{Subsec Notation}). The proof consists of three steps, each taking motivation from a different branch of mathematics. The first step of the proof generalizes the construction of the loop space of the circle from the category of spaces to $(\infty,1)$-categories (\cref{Sec The Loop Space of the Circle}). In the next step, we use our knowledge of elementary toposes to realize that we can construct a natural number object in the {\it underlying elementary topos} of our $(\infty,1)$-category (\cref{Subsec Constructing a Natural Number Object in the Truncation}). 
 \par 
 Finally, we want to show that this implies the $(\infty,1)$-category itself has a natural number object, however, here we encounter a serious complication. Concretely, there are several ways of characterizing natural number objects in an elementary topos: {\it Lawvere}, {\it Freyd} and {\it Peano}. The definition of a Peano and Freyd natural number object generalize with minimal effort (\cref{Subsec Peano and Lawvere Higher}), whereas Lawvere natural number objects do not, as we shall explain.
 \par 
 In the setting of elementary toposes the proof that a Peano natural number object is a Lawvere natural number object relies on the fact that elementary toposes are models of type theories \cite[Proposition D4.3.15]{johnstone2002elephantsii}. Indeed, the key step in \cite[Theorem D5.1.2]{johnstone2002elephantsii} is the construction of a subobject of the natural number object via a term in type theory. This construction cannot simply be repeated with $(\infty,1)$-categories because the connection with homotopy type theory is not yet properly understood. We might hope to translate the term manually in an $(\infty,1)$-category. However, as it involves identity types, its translation would not even be a subobject, as identity types in $(\infty,1)$-categories translate to path objects. Hence, there is simply no way to recover the proof in \cite[Theorem D5.1.2]{johnstone2002elephantsii} either directly or via translation.
 \par 
 Our next hope might be to instead use the fact that natural number objects have been studied in homotopy type theory \cite{ags2017initialalg} and manually translate a proof that Peano natural number objects give us Lawvere natural number objects \cite{shulman2020gitnno} into the  $(\infty,1)$-categorical setting. 
 The idea of the proof is to show that the type of morphisms out of the natural number object is a retract of the type of ``partially defined maps".
 However, the construction of this retract (and in particular \cite[Lines 983-1009]{shulman2020gitnno}) uses certain features of the internal language of homotopy type theory, that can only be translated if we have strict univalent universes (which only exist in some $(\infty,1)$-categories \cite{shulman2019inftytoposunivalent}).
 \par
 Given that none of the direct solutions have worked, we will improvise, meaning we will use the general idea of the proof in \cite{shulman2020gitnno} with some major changes. First, our construction of the {\it space of partial maps} (given in \cref{Subsec Space of Partial Maps}) will be explicitly parameterized over the natural number object. This way we can construct the desired retract also in an $(\infty,1)$-categorical setting, which we do in \cref{Subsec Total vs. Partial Maps}. Moreover, in order to prove the contractibility of the space of partial maps (\cref{The Part Contractible}) we will make explicit use of the $(\infty,1)$-categorical analogue of the {\it object of contractibility}, which we will introduce in \cref{Subsec The Internal Object of Contractibility}. Finally, given that our construction of the partial maps is parameterized over the natural numbers, we need to take an additional step and prove that the various fibers are compatible, which we do in \cref{Prop Partial Maps Parametrized}, again by using the object of contractibility. It should be noted that, if the connection between homotopy type theory and elementary $(\infty,1)$-topos theory is clarified, we could formalize this proof into an alternative proof to the one in \cite{shulman2020gitnno}.
 \par 
 We can use the existence of natural number objects in elementary $(\infty,1)$-toposes to study infinite colimits (\cref{Subsec Infinite Colimits in an Elementary Higher Topos}, \cref{Subsec Internal Infinite Coproducts and Sequential Colimits}). We can also use it deduce that not every elementary topos can be lifted to an elementary $(\infty,1)$-topos, leading to the natural question whether every elementary topos with natural number object can be in fact lifted (\cref{Subsec Not every Elementary Topos lifts to an Elementary Higher Topos}).
 The existence of natural number objects has been used in subsequent work to study truncations in an elementary $(\infty,1)$-topos \cite{rasekh2018truncations}. Moreover, using a filter construction for $(\infty,1)$-categories, we can now construct elementary $(\infty,1)$-toposes with non-standard natural number objects \cite{rasekh2020filterquotient}, which is a noteworthy as it cannot happen in a Grothendieck $(\infty,1)$-topos (\cref{Subsec Infinite Colimits in an Elementary Higher Topos}).

 \subsection{Background} \label{Subsec Background}
 Throughout this paper we use many ideas motivated by elementary $(\infty,1)$-toposes as defined in \cite{rasekh2018elementarytopos} and, in particular, the concept of {\it descent}. Moreover, we use the basics of $(\infty,1)$-category theory \cite{rezk2001css,lurie2009htt,riehlverity2017inftycosmos}. 
In \cref{Sec The Loop Space of the Circle}, we will use some basic observations from classical algebraic topology. We make extensive use of elementary topos theory and in particular results related to natural number objects and finite cardinals in an elementary topos \cite{johnstone2002elephanti,johnstone2002elephantsii}. Finally, we use some ideas from homotopy type theory and in particular the object of contractibility as studied in \cite{shulman2015elegantunivalence}.
 
 \subsection{\texorpdfstring{$(\infty,1)$}{(oo,1)}-Categorical Notation $\&$ Convention} \label{Subsec Notation}
 Throughout this whole paper $\E$ is a {\it finitely bicomplete locally Cartesian closed} $(\infty,1)$-category that has a {\it subobject classifier} and satisfies {\it descent} \cite[6.1.3]{lurie2009htt}.
 Examples of such $(\infty,1)$-categories include Grothendieck $(\infty,1)$-toposes \cite{lurie2009htt,rezk2010toposes}, but also filter product elementary $(\infty,1)$-toposes \cite{rasekh2020filterquotient}.
 Notice if $\E$ satisfies these conditions, then $\tau_0\E$ is also locally Cartesian closed with subobjects classifier, meaning it is an {\it elementary topos} \cite{maclanemoerdijk1994topos}. We will call $\tau_0\E$ the {\it underlying elementary topos} of $\E$.
 We will need some additional assumptions for $\E$ in \cref{Sec Applications}. 
 \par 
 For two objects $X,Y$ we denote the {\it mapping space} by $\Map(X,Y)$, whereas the space of equivalences is denoted $\Eq(X,Y)$. 
 Moreover, the {\it $(\infty,1)$-category of arrows} is denoted by $\Arr(\E)$, whereas the subcategory with same objects, but morphisms pullback squares is denoted by $\O_\E$.
 \par 
 For a given $(\infty,1)$-category $\E$, we denote underlying groupoid (the right adjoint to the inclusion of spaces) by $\E^{core}$, and the groupoidification (the left adjoint to the inclusion of spaces) by $\E^{grpd}$.
 Finally, we will use the fact that the target functor $t:O_\E \to \E$ is the {\it right fibration} that classifies the space valued presheaf that takes an object $B$ to $(\E_{/B})^{core}$ \cite[6.1.1]{lurie2009htt}.
 \par
  An object $X$ in $\E$ is $n$-truncated if for all objects $Y$, $\Map(Y,X)$ is an $n$-truncated space. In particular, $X$ is $(-1)$-truncated if and only if it is mono, if and only if the diagonal map $\Delta: X \to X \times X$ is an equivalence. We denote the full subcategory of $n$-truncated objects by $\tau_n\E$ and note the inclusion $\tau_n\E \hookrightarrow \E$ is limit preserving.
 \par
 We denote the final object in $\E$ with $1_\E$ and similarly the initial object with $\emptyset_\E$. Moreover, we will also denote the circle in $\E$ by $S^1_\E$ (\cref{Def Circle}), 
 whereas the circle in spaces is simply denoted by $S^1$.
 This way we avoid any confusion between objects in $\E$ and spaces.
 
 \subsection{Acknowledgments}
 I want to thank Egbert Rijke and Mike Shulman for making me aware that this result holds in homotopy type theory.
 I want to specially thank Mike Shulman for finding an error in a previous version of this paper and suggesting 
 an alternative approach that has led to the correction.
 I also want to thank the referee for many helpful suggestions that have significantly improved the exposition and organization of the paper.
 I would also like to thank the Max-Planck-Institut f{\"u}r Mathematik for its hospitality and financial support.
 
\section{The Loop Space of the Circle} \label{Sec The Loop Space of the Circle}
 The goal of this section is to gain a thorough understanding of the loop object of the circle.
 In the next section we will use this knowledge to construct natural number objects.
 
 \subsection{Descent and the Circle}
 
 \begin{defone}
   Let $\DA$ be the category with two objects and two maps which both start and end with the same objects.
   Informally we can depict it as 
    \begin{tikzcd}[row sep=0.5in, column sep=0.5in]
     \cdot \arrow[r, shift left=0.05in] \arrow[r, shift right=0.05in] & \cdot 
    \end{tikzcd}.
 \end{defone}

 \begin{defone} \label{Def Circle}
  Let $S^1_\E$ be the colimit of the final map $\DA \to \E$, meaning it is following coequalizer
  \begin{center}
   \begin{tikzcd}[row sep=0.5in, column sep=0.5in]
    1_\E \arrow[r, shift left=0.05in, "id_{1_\E}"] \arrow[r, shift right=0.05in, "id_{1_\E}"'] & 1_\E \arrow[r, "i_{\E}"] & S^1_{\E}
   \end{tikzcd}.
  \end{center}
  Moreover, define $\Omega S^1_\E$ as the pullback $\Omega S^1 = 1_\E \times_{S^1_\E} 1_\E$.
 \end{defone}
 
 We will prove the following facts about $\Omega S^1_\E$. It is $0$-truncated (\cref{The Loops Zero Truncated}), it is an initial algebra (\cref{Prop Initial Loop Algebra}), it is a group object (\cref{The Loop is a Group}) and we have an isomorphism 
  $\Omega S^1_\E \cong \coprod_n \Omega S^1_\E$ (\cref{Prop Omega Infinite}).
 
  In order to prove these we need to better understand the object $S^1_\E$, which requires us to study the concept of {\it descent}.
 It has been mostly studied in the context of a higher topos \cite[Section 6.1.3]{lurie2009htt}.
 We will only review the aspects of descent we need in the coming proofs.
 \par 
 Let $I$ be a finite $(\infty,1)$-category. 
 Then the diagonal map $\Delta_I: \E \to \E^I$ has a left adjoint that sends each diagram to the colimit 
 $\colim_I: \E^I \to \E$.
 Let $F: I \to \E$ be a fixed diagram in $\E$. Then we get a composition map 
 $\E^I_{/F} \to \E^I \to \E$.
  We can restrict $\E^I_{/F}$ to certain natural transformations. 
  A natural transformation $G: \Delta[1] \times I \to \E$ is {\it Cartesian} if for each map $\Delta[1] \to I$
  the restriction map $\Delta[1] \times \Delta[1] \to \E$ is a pullback square. 
  Let $F: I \to \E$ be a diagram in $\E$. We define $(\E^I_{/F})^{Cart}$ as the subcategory of $\E^I_{/F}$ with the same objects, but morphisms 
  Cartesian natural transformations. 

 \begin{exone} \label{Ex Cart over Final}
  Let $F: I \to \E$ be the final diagram. This means that $F$ factors through the map $1_\E: * \to \E$ that maps the point to the 
  final object in $\E$. In this case a Cartesian diagram $G: I \to \E$ over $F$ is a diagram that lifts to a diagram 
  $G: I^{grpd} \to \E$ (\cref{Subsec Notation}). 
 \end{exone}

 The projection map $(\E^I_{/F})^{Cart} \to \E$ is a right fibration, meaning that it models a contravariant functor in spaces 
 (for a detailed discussion of right fibrations see \cite{lurie2009htt,rasekh2017left,riehlverity2018elements}).
 We are finally in a position to state the descent condition we need.
 
 \begin{theone}
  The right fibration $(\E^I_{/F})^{Cart} \to \E$ is representable. Concretely the map $id: F \to F$ is a final object in 
  $(\E^I_{/F})^{Cart}$. The object $F$ maps to $\colim_I F$ in $\E$. This implies that we have an equivalence of right fibrations
  \begin{center}
   \begin{tikzcd}[column sep=0.25in, row sep=0.25in]
    \E_{/\colim_I F} \arrow[dr, "\pi"'] \arrow[rr, "\simeq"] & & (\E^I_{/F})^{Cart} \arrow[dl, "\colim_I"] \\
    & \E &
   \end{tikzcd}
  \end{center}
 \end{theone}

 \begin{remone}
  Intuitively we have following equivalence. Each map $Y \to \colim_I F$ can be pulled back to a Cartesian natural transformation diagram $G_Y \to F$.
  On the other side a Cartesian natural transformation over $F$, $G \to F$, gives us a map of colimits $\colim_I G \to \colim_I F$.
  Descent tells us that these two actions are inverses of each other. 
 \end{remone}

 We will use the descent property to analyze $\E_{/S^1_\E}$. Recall that $S^1_\E$ is a colimit of
 the final diagram of shape $\DA$. Thus using descent we get an equivalence 
 $\E_{/S^1_\E} \simeq ((\E^{\DA})_{/1_{\E}})^{Cart}$
 By \cref{Ex Cart over Final} a Cartesian natural transformation is a diagram out of the groupoidification of $\DA$, which 
 is the space $S^1$. Thus we get an equivalence 
 \begin{equation} \label{eq:descent}
 	\E_{/S^1_\E} \xrightarrow{ \ \ \simeq \ \ } \E^{S^1}.
 \end{equation}

  It is helpful to have some intuition on this equivalence. A map $X \to S^1_\E$ induces a diagram 
  \begin{center}
   \begin{tikzcd}[row sep=0.3in, column sep=0.5in]
    F \arrow[r, shift left=0.05in, "\simeq" near start, "e_1" near end] 
      \arrow[r, shift right=0.05in, "\simeq"' near start, "e_2"' near end] \arrow[d] 
    \arrow[dr, phantom, "\ulcorner", very near start] & 
    F \arrow[d] \arrow[r] \arrow[dr, phantom, "\ulcorner", very near start] & 
    X \arrow[d]\\
    1_\E \arrow[r, shift left=0.05in, "id_{1_\E}"] \arrow[r, shift right=0.05in, "id_{1_\E}"'] & 1_\E \arrow[r] & S^1_\E
   \end{tikzcd}
  \end{center}
  Thus we get an object $F$ along with two equivalences $e_1,e_2: F \to F$. Choosing a composition for $e_1, (e_2)^{-1}$ we get a self equivalence $e_1(e_2)^{-1}:F \to F$, which corresponds to a functor $S^1 \to \E$, as $S^1$ is the higher groupoid with one object and one self-equivalence.
  \par 
  On the other side every object $F$ with self equivalence $e$ gives us a diagram
  \begin{center}
   \begin{tikzcd}[row sep=0.3in, column sep=0.5in]
    F \arrow[r, shift left=0.05in, "\simeq" near start, "e" near end] 
      \arrow[r, shift right=0.05in, "\simeq"' near start, "id_F"' near end] \arrow[d] 
    \arrow[dr, phantom, "\ulcorner", very near start] & 
    F \arrow[d] \arrow[r] \arrow[dr, phantom, "\ulcorner", very near start] & 
    \Coeq(e,id_F) \arrow[d]\\
    1_\E \arrow[r, shift left=0.05in, "id_{1_\E}"] \arrow[r, shift right=0.05in, "id_{1_\E}"'] & 1_\E \arrow[r] & S^1_\E
   \end{tikzcd}
  \end{center}
  So we get a map $\Coeq(e,id_F) \to S^1_\E$. As a result of descent we know that these two maps are equivalences 
  and in fact inverses of each other.

  Notice one important example of the descent condition. By the equivalence of categories (\ref{eq:descent}), the map $i_\E :1_\E \to S^1_\E$ 
  corresponds to an object in $\E$ along with a self-equivalence of that object.
  We already realized that the object is $\Omega S^1_\E$. We will name the corresponding self-equivalence $s: \Omega S^1_\E \to \Omega S^1_\E$.

 \subsection{\texorpdfstring{$\Omega S^1_\E$}{S} is $0$-truncated:} 
 In this part we want to prove that $\Omega S^1_\E$ is $0$-truncated (\cref{Subsec Notation}).
 In order to prove the desired result, we need following technical lemma.
 
 \begin{lemone} \label{Lemma SOne truncated}
 	Assume we have following diagram in $\E$
 	\begin{center}
 		\begin{tikzcd}[row sep=0.35in, column sep=0.6in]
 			\hat{X} \arrow[r, shift left=0.05in, "\simeq" near start, "f" near end] 
 			\arrow[r, shift right=0.05in, "\simeq"' near start, "id_{\hat{X}}"' near end] \arrow[d, "g_1", bend left = 20] \arrow[d, "g_2"', bend right=20] & 
 			\hat{X} \arrow[d, "g_1", bend left = 20] \arrow[d, "g_2"', bend right=20] \arrow[r, "p"]  & 
 			\Coeq(f,id_{\hat{X}}) \arrow[d, "id"] \\
 			\hat{X} \arrow[r, shift left=0.05in, "\simeq" near start, "f" near end] 
 			\arrow[r, shift right=0.05in, "\simeq"' near start, "id_{\hat{X}}"' near end] & \hat{X} \arrow[r, "p"] & \Coeq(f,id_{\hat{X}})
 		\end{tikzcd}
 	\end{center}
    where the horizontal diagrams are coequalizers and the horizontal maps are equivalences. 
 	Then the space of equivalences $\Eq(g_1,g_2)$ is empty or contractible. 
 \end{lemone}
 
 \begin{proof}
 	 In order to prove this we need some notation.
 	First of all, we can think of the two right hand squares 
 	\begin{equation}\label{eq:squares}
 		\comsq{\hat{X}}{\Coeq(f,id_{\hat{X}})}{\hat{X}}{\Coeq(f,id_{\hat{X}})}{p}{g_1}{id}{p}, 
 		\comsq{\hat{X}}{\Coeq(f,id_{\hat{X}})}{\hat{X}}{\Coeq(f,id_{\hat{X}})}{p}{g_2}{id}{p}
 	\end{equation}
 	as cones over the diagram $X \rightarrow \Coeq(f,id) \leftarrow \Coeq(f,id)$. 
   Let $P$ be the category given by the poset 
   \begin{center}
   	\begin{tikzcd}[row sep=0.3in, column sep=0.3in]
   		0 \arrow[rr] \arrow[dr] & & 1 \arrow[dl] \arrow[r] &  3 \arrow[d] \\ 
   		 &2 \arrow[rr] & & 4 
   	\end{tikzcd}
   \end{center}
   and denote the unique morphism from $i \to j$ (if it exists) by $ij$. Observe that $P$ is a pushout of the commutative triangle $012$ and the square $1234$ glued along the morphism $12$. 
   
   Let $\Cone(g_1,g_2)$ be the space of diagrams $P \to \E$ that take the square formed by $0234$ to the left hand square in \cref{eq:squares} and 
   the square formed by $1234$ to the right hand square in \cref{eq:squares}. We can restrict such a diagram to the triangle formed by $012$, which has two properties. First, the restriction is an equivalence as the image of the square $1234$ is predetermined in $\Cone(g_1,g_2)$. Second, the triangle obtained by restricting along $012$ is such that the image of $02$ is $g_1$ and the image of $12$ is $g_2$. 
   Thus we have proven that there is a trivial fibration
   $$U: \Cone(g_1,g_2) \to \Map_{/X}(g_1: X \to X, g_2:X \to X).$$
   Now, by descent, the square formed by $g_2$ and $p$ is a pullback square which means the cone formed by $g_2$ 
   is a final object in the category of cones and thus 
   $\Cone(g_1,g_2)$ is contractible.
   \par 
   Finally, notice that $\Eq(g_1,g_2)$ corresponds to the subspace of $\Map_{/X}(g_1,g_2)$ consisting of maps $f: X \to X$ over $X$, such that $f$ is equivalent to the identity. 
	The fact that $U$ is a trivial fibration implies that this inclusion lifts to an inclusion
   \begin{center}
   	\begin{tikzcd}[row sep=0.3in, column sep=0.3in]
   		& Cone(g_1,g_2) \simeq * \arrow[d, "U", twoheadrightarrow]\\
   		Eq(g_1,g_2) \arrow[r, hookrightarrow] \arrow[ur, dashed, "I", hookrightarrow] & \Map_{/\hat{X}}(g_1,g_2) 
   	\end{tikzcd}
   \end{center}
   which implies that $\Eq(g_1,g_2)$ is either empty or contractible.
\end{proof}

 	In order to get more intuition for the proof, we can visualize the map $I$ as follows:
 		\begin{center}
 		\begin{tikzcd}[row sep=0.1in, column sep=0.2in]
 			X \arrow[rrrr, ""{name=A, below}, "id"] \arrow[ddrr, "g_1"'] & & & & X \arrow[ddll, "g_2"] & & 
 			X \arrow[ddr, "g_1"', bend right=20] \arrow[drr, "id", ""{name=B, below}] \arrow[rrrrrr, "p", ""{name=C, below}, bend left=20] 
 			\arrow[ddrrrrrr, phantom, "\ulcorner", very near start]
 			
 			& & & & & & \Coeq(f,id) \arrow[dd, "id"] \\ 
 			& & & & \strut \arrow[rr, mapsto, "I"] &  & \strut & & X \arrow[urrrr, "p"', bend right=20] \arrow[dl, "g_2", bend left=20] 
 			\arrow[drrrr, phantom, "\ulcorner", very near start] & & & & \\
 			& & X & & & & & X \arrow[rrrrr, "p"'] & & & & & \Coeq(f,id)
 			\arrow[Rightarrow, from=A, to=3-3, "\blueH", phantom]
 			\arrow[Rightarrow, from=B, to=3-8, "\blueH", phantom]
 			\arrow[Rightarrow, from=C, to=2-9, "\bluepH", phantom]
 		\end{tikzcd}.
 	\end{center}

 \begin{theone} \label{The Loops Zero Truncated}
  $\Omega S^1_\E$ is $0$-truncated.
 \end{theone}
 
 \begin{proof}
 	As $\Omega S^1$ is defined via pullback (\cref{Def Circle}), for any object $Y$, we have an equivalence of mapping spaces 
 	\begin{equation}\label{eq:loops}
 		\Map(Y,\Omega S^1_\E) \simeq \Omega \Map(Y,S^1_\E).
 	\end{equation}
 	Hence, it suffices to prove that $\Map(-,S^1_\E)$, or equivalently the right fibration $\pi: \E_{/S^1_\E} \to \E$, is $1$-truncated.
 	However, by descent, we know the right fibration $\pi: \E_{/S^1_\E} \to \E$ is equivalent to $colim: \E^{S^1} \to \E$. Hence, it suffices to prove that its fiber is $1$-truncated. However, this is precisely the statement of \cref{Lemma SOne truncated}.
 \end{proof}

  Notice that the result still holds if we take a bouquet of circles. Concretely let $\coprod_{S} 1_\E$, where $S$ is a finite set.
  Then $\bigvee_S S^1_\E$, defined as the pushout of the map $\coprod_{S} 1_\E \to 1_\E$ along itself, is also $1$-truncated.

 \subsection{\texorpdfstring{$\Omega S^1_\E$}{S} is the free Algebra generated by the Final Object}
 We say an object $X$ in $\E$ is an $\mathbb{A}$-{\it algebra} if it comes with an auto-equivalence $X \to X$ and a map from the initial object $A \to X$.
 Moreover, a map of $\mathbb{A}$-algebras simply commutes with these two maps.
 We want to prove that $A \xrightarrow{id \times o} A \times \Omega S^1_\E \xrightarrow{id \times s} A \times \Omega S^1_\E$ is the initial $\mathbb{A}$-algebra in $\E_{A/}$.
 \begin{lemone}
  The forgetful map $U: \E^{S^1} \to \E$ has a left adjoint.
 \end{lemone}
 
 \begin{proof}
  By descent we have an equivalence $\E_{/ S^1_\E} \to \E^{S^1}$. The composition map $\E_{/ S^1_\E} \to \E$ corresponds to pulling back 
  along the map $o: 1_\E \to S^1_\E$. But the map $o^*: \E_{/ S^1_\E} \to \E$ has an obvious left adjoint, namely 
  $o_! : \E \to \E_{/ S^1_\E}$. Thus the forgetful functor has a left-adjoint.
 \end{proof}
 
 \begin{remone} \label{Rem Free Alg on X}
  We realized that the map $o^*$ corresponds to the forgetful functor $U$. What does the composition map $o_!$ correspond to
  when we think about it as a map $\E \to \E^{S^1}$? An object $X$ is taken to the map $X \to 1_\E \to S^1_\E$. 
  Using out previous pullback construction we get following pullback diagram.
  \begin{center}
   \begin{tikzcd}[row sep=0.3in, column sep=0.6in]
     X \times \Omega S^1_\E \arrow[r, shift left=0.05in, "id_X \times s"] 
     \arrow[r, shift right=0.05in, "id_X \times id_{\Omega S^1_\E}"'] \arrow[d] 
     & 
     X \times \Omega S^1_\E \arrow[d] \arrow[r] \arrow[dr, phantom, "\ulcorner", very near start] & 
     X \arrow[d]
     \\
    \Omega S^1_\E \arrow[r, shift left=0.05in, "s"] 
      \arrow[r, shift right=0.05in, "id_{\Omega S^1_\E}"'] \arrow[d] 
    \arrow[dr, phantom, "\ulcorner", very near start] & 
    \Omega S^1_\E \arrow[d] \arrow[r] \arrow[dr, phantom, "\ulcorner", very near start] & 
    1_\E \arrow[d, "o"]
    \\
    1_\E \arrow[r, shift left=0.05in, "id_{1_\E}"] \arrow[r, shift right=0.05in, "id_{1_\E}"'] & 1_\E \arrow[r] & S^1_\E
   \end{tikzcd}
  \end{center}
  Thus the map $\E \to \E^{S^1}$ takes an object $X$ to the equivalence $id_X \times s: X \times \Omega S^1_\E \to X \times \Omega S^1_\E$.
 \end{remone}

  Let $A$ be an object and  $\E_{A /}$ be the category of objects under $A$. Then we can define the category of equivalences under $A$ 
  as the pullback $\E_{A /} \times_{\E} \E^{S^1}$ induced by the forgetful map $U: \E^{S^1} \to \E$.
  An object in this category is a chain $A \xrightarrow{ \ \ x \ \ } X \xrightarrow{ \ \ f \ \ } X$ where $f$ is an equivalence.
 
 \begin{propone} \label{Prop Initial Loop Algebra}
  The category $\E_{A /} \times_\E \E^{S^1}$ has an initial object.
 \end{propone}
 
 \begin{proof}
  We have the following pullback diagram
  \begin{equation} \label{eq:algebra square}
   \pbsq{\E_{A /} \times_\E \E^{S^1}}{\E_{ A /}}{\E^{S^1}}{\E}{}{}{}{U}
  \end{equation}
  As $\E_{A /}$ has an initial object, according to \cite[2.3]{gepnerkock2017univalence}, the pullback in \ref{eq:algebra square} has an initial object if $U$ has a left adjoint $L$.
  Moreover, in that case the initial object is the unit map of the adjunction $u_A: A \to UL(A)$. 
  However, we have just proven the existence of a left adjoint in the previous lemma. Thus $\E_{A /} \times_\E \E^{S^1}$ has the initial object 
  $A \xrightarrow{ \ \ id_A \times o \ \ } A \times \Omega S^1_\E \xrightarrow{ \ \ id_A \times s \ \ } A \times \Omega S^1_\E.$
 \end{proof}

  Concretely, being initial means that for any other object $A \xrightarrow{ \ \ x \ \ } X \xrightarrow{ \ \ f \ \ } X$
  There is a unique (up to contractible choice) map $g: A \times \Omega S^1_\E \to X$ filling the diagram below.
  \begin{center}
   \begin{tikzcd}[row sep=0.15in, column sep=0.25in]
    & A \times \Omega S^1_\E \arrow[r, "id_A \times s"] \arrow[dd, "g", dashed] & A \times \Omega S^1_\E \arrow[dd, "g", dashed] \\
    A \arrow[ur, "id_A \times o"] \arrow[dr, "x"] & & \\
    & X \arrow[r, "f"] & X
   \end{tikzcd}
  \end{center}
  In particular if $A = 1_\E$ then the initial object in $\E_{1_\E /} \times_\E \E^{S^1}$
  is of the form 
  $$ 1_\E \xrightarrow{ \ \ o \ \ } \Omega S^1_\E \xrightarrow{ \ \ s \ \ } \Omega S^1_\E$$
  which implies that $\Omega S^1_\E$ is an initial $\mathbb{A}$-algebra.

  \subsection{\texorpdfstring{$\Omega S^1_\E$}{S} is a Group Object}
 Having shown $\Omega S_\E^1$ is $0$-truncated we can now easily prove the following. 
 
 \begin{theone} \label{The Loop is a Group}
  $\Omega S^1_\E$ is a group object in $\E$.
 \end{theone}
 
 \begin{proof}
  As we have shown in \ref{eq:loops}, we have an equivalence $\Map(X, \Omega S^1_\E) \simeq \Omega \Map(X, S^1_\E)$.
  Thus the space $\Map(X, \Omega S^1_\E)$ is a loop space, which implies that $\pi_0(\Map(X, \Omega S^1_\E))$
  is a group. However, as $\Omega S^1_\E$ is $0$-truncated we know that 
  $\pi_0(\Map(X, \Omega S^1_\E)) = \Map(X, \Omega S^1_\E)$, which implies that $\Map(X, \Omega S^1_\E)$ is itself
  a group. This proves that $\Omega S^1_\E$ is group object in $\tau_0 \E$ and more generally in $\E$.
 \end{proof}
 
 It is instructive to concretely understand the group structure on the set $\Hom(X, \Omega S^1_\E)$, which corresponds to loops around the point $X \to 1_\E \xrightarrow{ \ i \ } S^1_\E$ in the space $\Map(X,S^1)$, for which we need a change of perspective on the map $X \to 1_\E \to S^1_\E$.
 By the arguments in \cref{Rem Free Alg on X} we know that the map corresponds to the equivalence
 $id_X \times s: X \times \Omega S^1_\E \to X \times \Omega S^1_\E$.
 From this perspective a loop is simply a commutative diagram 
 \begin{equation} \label{eq:loop correspondence}
  \begin{tikzcd}[row sep=0.3in, column sep=0.3in]
    X \times \Omega S^1_\E \arrow[r, "id_X \times s"] \arrow[d, "\simeq"', "f"] & X \times \Omega S^1_\E \arrow[d, "\simeq"', "f"] \\
    X \times \Omega S^1_\E \arrow[r, "id_X \times s"] & X \times \Omega S^1_\E
  \end{tikzcd}
 \end{equation}
 The group operation on $\Hom(X, \Omega S^1_\E)$ then corresponds to composing two squares vertically.
 
 \begin{exone}
  Let us see how this example looks like in the classical setting of spaces. In spaces we know that $\Omega S^1 = \mathbb{Z}$. 
  The map $s: \mathbb{Z} \to \mathbb{Z}$ then corresponds to the successor map, which takes $n$ to $n+1$.
  \par 
  We want to classify the automorphisms of $f: \mathbb{Z} \to \mathbb{Z}$ that commute with $s$. A simple exercise shows that 
  if $f$ commutes with $s$ then there exists an integer $m$ such that $f(n) = n + m$. This gives us a bijection between 
  $\mathbb{Z}$ and $\Omega S^1$.
 \end{exone}

  The map $s: \Omega S^1_\E \to \Omega S^1_\E$ corresponds to composing the square in \cref{eq:loop correspondence} with the square
  \begin{center}
   \comsq{X \times \Omega S^1_\E}{X \times \Omega S^1_\E}{X \times \Omega S^1_\E}{X \times \Omega S^1_\E}{
   id_X \times s}{id_X \times s}{id_X \times s}{id_X \times s}
  \end{center}
  which means it takes $f$ to $(id_X \times s)f$.

 Now that we know $\Omega S^1_\E$ is a group object we can think of the map $s:\Omega S^1_\E \to \Omega S^1_\E$ as a map 
 of groups and so we might wonder how ``free" this map is. 
 
 \begin{lemone} \label{Lemma S No Fixed Points}
  In the equalizer diagram 
  \begin{center}
   \begin{tikzcd}[row sep=0.5in, column sep=0.5in]
    A \arrow[r] & \Omega S^1_\E \arrow[r, shift left=0.05in, "s"] 
    \arrow[r, shift right=0.05in, "id"'] & \Omega S^1_\E
   \end{tikzcd}
  \end{center}
   we have $A \cong \emptyset_\E$, the initial object.
 \end{lemone}
 
 \begin{proof}
  First, notice $A$ is also the equalizer of the two maps $id_A \times s,id: A \times \Omega S^1_\E \to A \times \Omega S^1_\E$. 
  By \cref{Rem Free Alg on X}, the map $A \to \Omega S^1_\E$, which is a point $\Map(A,\Omega S^1_\E) \simeq \Omega \Map(A,S^1_\E)$, corresponds to a diagram 
  \begin{center}
  \begin{tikzcd}[row sep=0.3in, column sep=0.3in]
    A \times \Omega S^1_\E \arrow[r, "id_A \times s"] \arrow[d, "\cong"', "f"] & A \times \Omega S^1_\E \arrow[d, "\cong"', "f"] \\
    A \times \Omega S^1_\E \arrow[r, "id_A \times s"] & A \times \Omega S^1_\E
  \end{tikzcd}.
 \end{center}
  such that $(id_A \times s)f \simeq f$. However, as the diagram is commutative, we also have $(id_A \times s)f \simeq f(id_A \times s)$,
  which implies that $f(id_A \times s) \simeq f$. This gives us following diagram 
  \begin{center}
   \begin{tikzcd}[row sep=0.3in, column sep=0.5in]
    A \times \Omega S^1_\E \arrow[r, shift left=0.05in, "id_A \times s"] \arrow[r, shift right=0.05in, "id"'] 
    & A \times \Omega S^1_\E \arrow[r, "id_A \times i_{\E}"] \arrow[d, "f"', "\cong"] & A \arrow[dl, dashed, "g"] \\
    & A \times \Omega S^1_\E
   \end{tikzcd}.
  \end{center}
  The fact that the isomorphism $f$ factors through $id_A \times i_\E$ implies that $id_A \times i_\E$ is an isomorphism. 
  Now we have following diagram 
  \begin{center}
   \begin{tikzcd}[row sep=0.3in, column sep=0.3in]
    A \times \Omega S^1_\E \arrow[r, shift left=0.05in, "id_A \times s"] 
      \arrow[r, shift right=0.05in, "id"'] \arrow[d, "\simeq"] 
    \arrow[dr, phantom, "\ulcorner", very near start] & 
    A \times \Omega S^1_\E \arrow[d, "\simeq"] \arrow[r] \arrow[dr, phantom, "\ulcorner", very near start] & 
    A \arrow[d]\\
    A \arrow[r, shift left=0.05in, "id_A"] \arrow[r, shift right=0.05in, "id_A"'] & A \arrow[r] & A \times S^1_\E
   \end{tikzcd}
  \end{center}
  By homotopy invariance of colimits, the map $A \to A \times S^1_\E$ is an equivalence. This by descent implies that we have an equivalence 
  $ \E_{/A} \simeq \E_{/A \times S^1_\E} \simeq (\E_{/A})^{S^1},$
  which means that $\E_{/A}$ does not have any non-trivial self-equivalences. In particular this means that $A \times S^1_\E \simeq S^1_\E$.
  \par 
  Let $\sigma : 1_\E \coprod 1_\E \to 1_\E \coprod 1_\E$ be the switch map. Then, $id_A \times \sigma \simeq id$ as we do not have any non-trivial automorphisms.
  However, $\Coeq(id_A \times \sigma, id) = A \times S^1 \simeq A$ and $\Coeq(id,id) \simeq A  \times S^1_\E \coprod A \times S^1_\E \simeq A \coprod A$. 
  Thus $A \simeq A \coprod A$. Let $\Omega$ be the subobject classifier. Then 
  $$\Sub(A) \cong \Map(A \coprod A, \Omega) \simeq \Map(A, \Omega) \times \Map(A,\Omega) \cong \Sub(A) \times \Sub(A)$$
  which implies that the diagonal $\Delta: \Sub(A) \to \Sub(A) \times \Sub(A)$ is an isomorphism of sets, which is only possible if $\Sub(A)$ is either empty or the point. 
  This means that $A$ is equivalent to its least subobject, which is the initial object $\emptyset_\E$.
  Hence $A \cong \emptyset_\E$ which finishes the proof.
 \end{proof}

 \subsection{Covering Spaces} 
  Recall from classical algebraic topology, that there is a map 
  $n: S^1 \to S^1$
  that wraps the circle around itself $n$ times. This covering map gives us following pullback square
  \begin{center}
  	\pbsq{ \{0 \comma 1 \comma ... \comma n-1 \} }{S^1}{*}{S^1}{}{}{n}{i}.
  \end{center}
  The goal is to replicate this process in $\E$.
  In order to do so we need additional notation.
  Let $[n] = \{ 0, 1,2, ..., n-1 \}$. Notice $[n]$ has $n$ elements.
  We define the map $s_n: [n] \to [n]$ that sends $i$ to $i+1$ if $i<n-1$ and $n-1$ to $0$.   
  
   Let $nS^1$ be the colimit of the following diagram 
  \begin{center}
  	\begin{tikzcd}[row sep=0.5in, column sep=0.5in]
  		\leb n \reb \arrow[r, shift left=0.05in, "id_{[n]}"] \arrow[r, shift right=0.05in, "s_n"'] & \leb n \reb  \arrow[r] & nS^1
  	\end{tikzcd}.
  \end{center}
  Notice that $nS^1$ is homotopy equivalent to $S^1$, but has a different cell structure.
  In particular, it has $n$ $0$-cells and $n$ $1$-cells.
   
  Using the colimit description, we can define a map $n_\E: nS^1_\E \to S^1_\E$ in $\E$ as follows:
  \begin{equation}\label{eq:ns pullback}
  	\begin{tikzcd}[row sep=0.3in, column sep=0.3in]
  		\ds \coprod_{[n]} 1_\E \arrow[r, shift left=0.05in, "s_n"] 
  		\arrow[r, shift right=0.05in, "id"'] \arrow[d] 
  		\arrow[dr, phantom, "\ulcorner", very near start] & 
  		\ds \coprod_{[n]} 1_\E \arrow[d] \arrow[r] \arrow[dr, phantom, "\ulcorner", very near start] & 
  		nS^1_\E \arrow[d, "n_\E"]\\
  		1_\E \arrow[r, shift left=0.05in, "id"] \arrow[r, shift right=0.05in, "id"'] & 1_\E \arrow[r] & S^1_\E
  	\end{tikzcd}.
  \end{equation}
  We can think of the induced map of colimits, $n_\E: nS^1_\E \to S^1_\E$, as the map that corresponds to the classical 
  "wrap-around map" in algebraic topology. 
 \begin{propone} \label{Prop Omega Infinite}
  Let $n$ be a natural number. Then $ \coprod_n \Omega S^1_\E \cong \Omega S^1_\E.$
 \end{propone}
 \begin{proof}
 	We can pull back the right hand square in \cref{eq:ns pullback} along the map $i_\E: 1_\E \to nS^1_\E$ to get the diagram
 	\begin{center}
 		\begin{tikzcd}[row sep=0.3in, column sep=0.3in]
 		\Omega S^1_\E \arrow[ddr, bend right=20] \arrow[r, "\simeq", phantom] & \ds \coprod_{[n]} \Omega S^1_\E \arrow[d] \arrow[r] \arrow[dr, phantom, "\ulcorner", very near start] & 1_\E \arrow[d, "i_\E"] \arrow[dd, bend left=45, "i_\E", ""{name=U}]\\
 		  & 	\ds \coprod_{[n]} 1_\E \arrow[r] \arrow[d] \arrow[dr, phantom, "\ulcorner", very near start] & nS^1_\E \arrow[d, "n_\E"] \\
 		 &	1_\E \arrow[r] & S^1_\E
 		 \arrow[from=U, to=2-3, phantom, "\simeq"]
 		\end{tikzcd}.
 	\end{center}
    The fact that we have an equivalence $n_\E i_\E \simeq i_\E$ implies, by homotopy invariance of pullbacks, that $\Omega S^1_\E \simeq \coprod_{[n]} \Omega S^1_\E$. However as both sides are actually $0$-truncated we get in fact an isomorphism 
    $ \Omega S^1_\E \cong \coprod_{[n]} \Omega S^1_\E$.
\end{proof}
  
\section{Constructing Peano Natural Number Objects} \label{Sec Constructing Peano Natural Number Objects}
 The goal of the this sections is to prove that every $(\infty,1)$-category $\E$ satisfying the conditions of \cref{Subsec Notation}
 has a Peano and Freyd natural number object.
 We use our analysis of $\Omega S^1_\E$ to give two constructions of natural number object $(\mathbb{N}_\E,o,s)$ in $\tau_0 \E$, a non-canonical one (\cref{prop:noncanonical nno}) and a canonical one (\cref{prop:canonical NNO}). 
 Then we finish this section by showing that our natural number object in $\tau_0 \E$ gives us a Peano natural 
 number object (\cref{Def Peano NNO}) 
 in $\E$ (\cref{Lemma EHT Peano NNO})  
 and a Freyd natural number object (\cref{Def Freyd NNO}) in $\E$ (\cref{Prop EHT Freyd NNO}).
 In the next section, we will then show that this object $(\mathbb{N}_\E,o,s)$ is also a Lawvere natural number object 
 (\cref{Def NNO Lawvere}) in $\E$.
 
 \subsection{Natural Number Objects in \texorpdfstring{$(\infty,1)$}{(oo,1)}-Categories}
 In this subsection we introduce natural number objects in $(\infty,1)$-categories and review important results regarding natural number objects in elementary toposes, as discussed in \cite{johnstone2002elephanti,johnstone2002elephantsii}.
 Throughout this section $\E$ is an $(\infty,1)$-category satisfying the conditions in \cref{Subsec Notation}.
 
 \begin{defone} \label{Def NNO Lawvere}
  A {\it Lawvere natural number object} in $\E$ is a triple $(\mathbb{N}_\E,o:1_\E \to \mathbb{N}_\E,s:\mathbb{N}_\E \to \mathbb{N}_\E)$ 
  such that $(\mathbb{N}_\E , o , s)$ is initial.
 \end{defone}
   
 \begin{remone} \label{Rem NNO Space}
  To make things more explicit, this is saying that for any other triple $(X,b:1_\E \to X,u:X \to X)$
  the space of maps $f: \mathbb{N}_\E \to X$ that make the following diagram commute is contractible
  \begin{center}
   \begin{tikzcd}[row sep=0.1in, column sep=0.3in]
    & \mathbb{N}_\E \arrow[r, "s"] \arrow[dd, "f", dashed] & \mathbb{N}_\E \arrow[dd, "f", dashed] \\
    1_\E \arrow[ur, "o"] \arrow[dr, "b"'] & & \\
    & X \arrow[r, "u"] & X
   \end{tikzcd}.
  \end{center}
  We can rephrase this by saying the limit of the diagram, which we denote by $\Ind(X,b,u)$, is contractible.
  \begin{equation}\label{eq:limit Lawvere ext}
  	 \begin{tikzcd}[row sep=0.3in, column sep=0.3in]
  	 	* \arrow[r, "b"] &[-0.15in] \Map_\E(1_\E,X) &[-0.1in] \Map_\E(\mathbb{N}_\E,X) \arrow[l, "o^*"'] \arrow[r, "(s^* \comma u_*)"]  &\Map_\E(\mathbb{N}_\E,X) \times \Map_\E(\mathbb{N}_\E,X) &[-0.1in] \Map_\E(\mathbb{N}_\E,X) \arrow[l, "\Delta"']
  	 \end{tikzcd}.
  \end{equation}
\end{remone}
 
 There are two alternative ways to define a natural number object.
 
 \begin{defone} \label{Def Freyd NNO}
  A {\it Freyd natural number object} is a triple $(\mathbb{N}_\E,o:1_\E \to \mathbb{N}_\E,s:\mathbb{N}_\E \to \mathbb{N}_\E)$ 
  such that the following two diagrams are colimit diagrams
  \begin{center}
    \begin{tikzcd}[row sep=0.3in, column sep=0.3in]
    \mathbb{N}_\E \arrow[r, shift left=0.05in, "id"] \arrow[r, shift right=0.05in, "s"'] & \mathbb{N}_\E  \arrow[r] & 1_\E
    \end{tikzcd},
   \begin{tikzcd}[row sep=0.3in, column sep=0.3in]
    \emptyset \arrow[d] \arrow[r] & \mathbb{N}_\E \arrow[d, "s"] \\
    1_\E \arrow[r, "o"] &\mathbb{N}_\E \arrow[ul, phantom, "\ulcorner", very near start]
   \end{tikzcd}.
  \end{center}
 \end{defone}
 
 \begin{defone} \label{Def Peano NNO}
  A {\it Peano natural number object}  is a triple $(\mathbb{N}_\E,o:1_\E \to \mathbb{N}_\E,s:\mathbb{N}_\E \to \mathbb{N}_\E)$ 
  that satisfies following conditions:
  \begin{enumerate}
   \item $s$ is monic.
   \item $o$ and $s$ are disjoint subobjects of $\mathbb{N}_\E$.
   \item Assume we have a subobject $\mathbb{N}_\E' \hookrightarrow \mathbb{N}$ that is closed under the maps $o$ and $s$, meaning we have a commutative diagram 
    \begin{center}
   \begin{tikzcd}[row sep=0.1in, column sep=0.3in]
    & \mathbb{N}_\E' \arrow[r, "s"] \arrow[dd, hookrightarrow] & \mathbb{N}_\E' \arrow[dd, hookrightarrow] \\
    1_\E \arrow[ur, "o"] \arrow[dr, "o"] & & \\
    & \mathbb{N}_\E \arrow[r, "s"] & \mathbb{N}_\E
   \end{tikzcd}.
  \end{center}
  Then the inclusion $\mathbb{N}_\E' \overset{\cong}{\hookrightarrow} \mathbb{N}_\E$ is an isomorphism.
 \end{enumerate}
 \end{defone}
 
 The next result, which states that the three definitions of natural number object coincide in the underlying elementary topos $\tau_0\E$, is a direct implication of the same result for elementary toposes \cite[Theorem D5.1.2]{johnstone2002elephantsii}.
 
 \begin{theone} \label{The ET NNOs Same}
  Let $\E$ be a $(\infty,1)$-category satisfying the conditions in \cref{Subsec Notation}. Then for any triple
  $(\mathbb{N}_\E,o: 1_\E \to \mathbb{N}_\E,s:\mathbb{N}_\E \to \mathbb{N}_\E)$ in the underlying elementary topos $\tau_0\E$ the following are equivalent:
  \begin{enumerate}
   \item $(\mathbb{N}_\E,o,s)$ is a Lawvere natural number object. 
   \item $(\mathbb{N}_\E,o,s)$ is a Freyd natural number object.
   \item $(\mathbb{N}_\E,o,s)$ is a Peano natural number object.
  \end{enumerate}
 \end{theone}
  
  This theorem gives us a helpful uniqueness result that we will use in the next sections.
  
  \begin{lemone} \label{lemma:nno unique}
  	Let $(\mathbb{N}_\E,o,s)$ be a Lawvere, Tierney or Peano natural number object in $\tau_0\E$. Similarly, let $(\mathbb{N}_\E',o',s')$ be a Lawvere, Tierney or Peano natural number object in $\tau_0\E$. Then there exists an isomorphism $f: \mathbb{N}_\E \to \mathbb{N}_\E'$ such that the following diagram commutes
  	 \begin{equation}\label{eq:all nno}
  		\begin{tikzcd}[row sep=0.1in, column sep=0.3in]
  			& \mathbb{N}_\E' \arrow[r, "s"] \arrow[dd, hookrightarrow, "f"] & \mathbb{N}_\E' \arrow[dd, hookrightarrow, "f"] \\
  			1_\E \arrow[ur, "o"] \arrow[dr, "o'"] & & \\
  			& \mathbb{N}_\E \arrow[r, "s'"] & \mathbb{N}_\E
  		\end{tikzcd}.
  	\end{equation}
  \end{lemone}
  \begin{proof}
  	By \cref{The ET NNOs Same}, $(\mathbb{N}_\E,o,s)$ and $(\mathbb{N}_\E',o',s')$ are Lawvere natural number objects. 
    Now, as $(\mathbb{N}_\E,o,s)$ is initial (\cref{Def NNO Lawvere}), there is a map $f: \mathbb{N}_\E \to \mathbb{N}_\E'$ making \cref{eq:all nno} commute. Finally, $(\mathbb{N}_\E',o',s')$ is initial as well and so $f$ is an isomorphism.
  \end{proof}
  
  Because of this lemma we will henceforth refer to {\it the} natural number object in $\tau_0\E$, as any two choices are isomorphic.
  
 \subsection{Constructing a Natural Number Object in the $0$-Truncation} \label{Subsec Constructing a Natural Number Object in the Truncation}
  Let $\E$ be as in \cref{Subsec Notation} and $\tau_0\E$ its underlying elementary topos. We will prove $\tau_0\E$ has a natural number object, using the fact that it includes $\Omega S^1_\E$ (as it is $0$-truncated by \cref{The Loops Zero Truncated}) and techniques from \cite{johnstone2002elephanti,johnstone2002elephantsii}.
 We will give two constructions: One being simpler but non-canonical, the other being more difficult but resulting in a canonical construction.
 
 \begin{propone} \label{prop:noncanonical nno}
 	Let $\E$ be an $(\infty,1)$-category satisfying the conditions in \cref{Subsec Notation}. 
 	Then the underlying elementary topos $\tau_0\E$ has a natural number object.
 \end{propone}
 
 \begin{proof}
 	Notice $\Omega S^1_\E \cong \Omega S^1_\E \coprod \Omega S^1_\E$ (\cref{Prop Omega Infinite}). As a result there exists an inclusion
 	$\iota_1: \Omega S^1_\E \to \Omega S^1_\E \coprod \Omega S^1_\E \cong \Omega S^1_\E .$
 	Similarly, there exists an inclusion map $\iota_2: \Omega S^1_\E \hookrightarrow\Omega S^1_\E$.
 	As $\tau_0 \E$ is an elementary topos, $\iota_1$ and $\iota_2$ are disjoint subobjects of $\Omega S^1_\E$.
 	Thus, in particular, the precomposition subobject $\iota_2 o : 1_\E \to \Omega S^1$ is disjoint from the 
 	subobject $\iota_1: \Omega S^1_\E \to \Omega S^1_\E$.
 	This means that the diagram 
 	$ 1_\E \xrightarrow{\iota_2 o} \Omega S^1_\E \xrightarrow{\iota_1 } \Omega S^1_\E$
 	satisfies the conditions of \cite[Corollary D.5.1.3]{johnstone2002elephantsii}, proving $\tau_0\E$ has a natural number object.
 \end{proof}
  
  Notice, the construction is indeed non-canonical, because it depends on the choice of isomorphism $\Omega S^1_\E \cong \Omega S^1_\E \coprod \Omega S^1_\E$, as we can observe from following examples.

 \begin{exone} \label{Ex Two NNOs}
  Let $\E$ be the $(\infty,1)$-category of spaces. Then $\Omega S^1 = \mathbb{Z}$, the set of integers.
  Now, we can choose the bijection $\mathbb{Z} \cong \mathbb{Z} \coprod \mathbb{Z}$ that identifies the first summand of $\mathbb{Z}$ with the even integers and the second summand with the odd integers. Then, following the construction in \cref{prop:noncanonical nno}, we get the natural number object 
  $\{2^n: n \in \mathbb{N} \} =  \{ 1, 2, 4, ... \}$, with
  successor map $s: \{2^n: n \in \mathbb{N} \} \to \{2^n: n \in \mathbb{N} \}$ being multiplication by $2$.
  
  Alternatively, we can choose the bijection $\mathbb{Z} \cong \mathbb{Z} \coprod \mathbb{Z}$, that identifies the first summand with the prime numbers (including $1$) in $\mathbb{Z}$ and the second summand with the remaining integers. Then, by \cref{prop:noncanonical nno}, we get the natural number object $\{ 1, 2, 3, 5, ... \}$ (the set of prime numbers and $1$) and the successor map assigns the next prime number.
 \end{exone}
 
 Given the example, we want to construct a canonical natural number object in $\Omega S^1_\E$.
 Let $(\mathbb{N}_\E^{can},o,s)$ be the smallest $(o,s)$-closed subobject of $(\Omega S^1,o,s)$ in $\tau_0\E$ (which exists by \cite[Lemma D5.1.1]{johnstone2002elephantsii}).
 
 \begin{propone} \label{prop:canonical NNO}
  The triple $(\mathbb{N}^{can}_\E,o,s)$ is a natural number object in $\tau_0\E$.
 \end{propone}

\begin{proof}
   We will prove that $(\mathbb{N}^{can}_\E,o,s)$ is a Peano natural number object in $\tau_0 \E$. 
   By \cref{Def Peano NNO}, we have to verify that $s: \mathbb{N}^{can}_\E \to \mathbb{N}^{can}_\E$ is mono, which follows immediately from the fact that $s: \Omega S^1_\E \to \Omega S^1_\E$ is an isomorphism, that any $(o,s)$-closed subobject of $\mathbb{N}^{can}_\E$ is equal to $\mathbb{N}^{can}_\E$, which holds by definition, and that in the pullback square 
     \begin{equation} \label{eq:initial}
   	\pbsq{U}{\mathbb{N}^{can}_\E}{1_\E}{\mathbb{N}^{can}_\E}{}{}{s}{o}
   \end{equation}
   $U = \emptyset_\E$, the initial object, which needs a more detailed analysis.
   \par 
   As $s: \mathbb{N}^{can}_\E \to \mathbb{N}^{can}_\E$ is mono the map $U \to 1_\E$ in \cref{eq:initial} is an inclusion and so $U$ is 
   a subobject of the final object. We want to prove that $U = \emptyset_\E$ i.e. $U$ is the smallest subobject. 
   Notice $U \to U \times \Omega S^1_\E \to U \times \Omega S^1_\E$ is also $0$-truncated and so also a diagram in $\tau_{ 0} \E$.
   Hence, by definition of $\mathbb{N}^{can}_\E$, $U \to U \times \mathbb{N}^{can}_\E \to U \times \mathbb{N}^{can}_\E$ 
   is the smallest closed subobject. Thus we get a diagram 
   \begin{center}
    \begin{tikzcd}[row sep=0.3in, column sep=0.5in]
     & U \times \mathbb{N}^{can}_\E \arrow[r, "id_U \times s"] \arrow[d, hookrightarrow, "id_U \times s"] & 
     U \times \mathbb{N}^{can}_\E \arrow[d, hookrightarrow, "id_U \times s"] \\
        U \arrow[ur, "id_U \times o", dashed] \arrow[dr, "id_U \times o"] \arrow[r, "id_U \times o"] & 
        U \times \mathbb{N}^{can}_\E \arrow[r, "id_U \times s"] \arrow[d, hookrightarrow, "id_U \times i"] & 
        U \times \mathbb{N}^{can}_\E \arrow[d, hookrightarrow, "id_U \times i"] \\
     & U \times \Omega S^1_\E \arrow[r, "id_U \times s"] & U \times \Omega S^1_\E
    \end{tikzcd}.
   \end{center}
  The fact that the map is initial implies that the map $id_U \times s: U \times \mathbb{N}^{can}_\E \to U \times \mathbb{N}^{can}_\E$ is an isomorphism.
  However, by \cref{Prop Initial Loop Algebra}, this means that we must have a unique map $g:U \times \Omega S^1_\E \to U \times \mathbb{N}^{can}_\E$ making the following diagram commute
  \begin{equation} \label{eq:g}
    \begin{tikzcd}[row sep=0.1in, column sep=0.25in]
     & U \times \Omega S^1_\E \arrow[r, "id_U \times s"] \arrow[dd, "g", dashed] & U \times \Omega S^1_\E \arrow[dd, "g", dashed] \\
     U \arrow[ur, "id_U \times o"] \arrow[dr, "id_U \times o"'] & & \\
     & U \times \mathbb{N}^{can}_\E \arrow[r, "id_U \times s"] & U \times \mathbb{N}^{can}_\E
    \end{tikzcd}
   \end{equation}
   which means that $U \to U \times \Omega S^1_\E \to U \times \Omega S^1_\E$ does 
   not have any non-trivial subobjects (as it is its own smallest subobject). 
   \par
   We can repeat everything we have done until now with the object $(\Omega S^1_\E, so,s)$ to conclude that $si: \mathbb{N}^{can}_\E \to \Omega S^1_\E$ is the minimal subobject and so $(\Omega S^1_\E \times U, so \times id_U, s \times id_U)$ does not have any non-trivial subobjects.
    \par 
    In particular this means that $si \times id_U$ and $i \times id_U$ are isomorphic, implying 
    the map $i \times id_U$ factors through the equalizer
    of $s \times id_U$ and $id_{\Omega S^1_\E} \times id_U$. However, in \cref{Lemma S No Fixed Points} 
    we showed that this equalizer is the initial object, $\emptyset_\E$. 
    This gives us a map $\mathbb{N}^{can}_\E \times U \to \emptyset_\E$, which means $\mathbb{N}^{can}_\E \times U \cong \emptyset_\E$. By the isomorphism in \cref{eq:g} this means
    $\Omega S^1_\E \times U \cong \emptyset$.
    Finally, by \cref{Rem Free Alg on X}, the coequalizer of the two maps 
    $s \times id_U, id_{\Omega S^1_\E} \times id_U: \Omega S^1_\E \times U \to \Omega S^1_\E \times U$ is $U$. 
    But $\Omega S^1_\E \times U \cong \emptyset_\E$, which implies that the coequalizer is $\emptyset_\E$.
    Thus $U \cong \emptyset_\E$, proving $\mathbb{N}_\E^{can}$ is a Peano natural number object.
   \end{proof} 
 
      Notice if we start with $\mathbb{Z}$ in spaces then the canonical natural number object is the actual set of natural numbers $\mathbb{N} \subset \mathbb{Z}$.
      This is a clear contrast to \cref{Ex Two NNOs}, where the resulting natural number objects are far more complicated subsets of $\mathbb{Z}$. 
 
 \subsection{Peano and Freyd Natural Number objects in \texorpdfstring{$(\infty,1)$}{(oo,1)}-Categories} \label{Subsec Peano and Lawvere Higher}
  In the final subsection we prove that $(\mathbb{N}_\E^{can},o,s)$ (\cref{prop:canonical NNO}) is also a Peano (and Freyd) natural number object 
  in $\E$. The case for Peano natural number objects is in fact even more general.
  
  \begin{lemone} \label{Lemma EHT Peano NNO}
  	Every Peano natural number object $(\mathbb{N}_\E,o,s)$ in $\tau_0\E$ is one in $\E$.
   \end{lemone}

   \begin{proof}
   	We need to prove that $(\mathbb{N}_\E,o,s)$ satisfies the three conditions in \cref{Def Peano NNO} in $\E$.
   	However, the inclusion $\tau_0\E \to \E$ is limit preserving (\cref{Subsec Notation}) and so a pullback square (mono map) in $\tau_0 \E$ remains a pullback square (mono map) in $\E$. Moreover, any $(o,s)$-closed subobject of $\mathbb{N}_\E$ in $\E$ is also one in $\tau_0\E$ and hence isomorphic to $\mathbb{N}$.
   \end{proof}

   \begin{propone} \label{Prop EHT Freyd NNO}
    The subobject $\mathbb{N}^{can}_\E$ of $\Omega S^1_\E$ is a Freyd natural number object in $\E$.
   \end{propone}

   \begin{proof}
    We have to show that the two diagrams in \cref{Def Freyd NNO} are colimit diagrams. The coproduct diagram follows from the fact that 
    it is a coproduct in $\tau_0 \E$ and coproducts in the subcategory of truncated objects are also coproducts in the original category. 
    However, this is generally not true for coequalizers and thus 
    we need a separate argument.
    \par 
    We have following diagram 
    \begin{center}
    \begin{tikzcd}[row sep=0.1in, column sep=0.25in]
     & \mathbb{N}^{can}_\E \arrow[r, shift left=0.05in, "s"] 
       \arrow[r, shift right=0.05in, "id"'] \arrow[dd, hookrightarrow] & \mathbb{N}^{can}_\E \arrow[dd, hookrightarrow] \arrow[r] & 
     U \arrow[dd, hookrightarrow]\\
     1_\E \arrow[rrrd, bend left=20, "id"] \arrow[dr, "o"'] \arrow[ur, "o"] \\
     & \Omega S^1_\E \arrow[r, shift left=0.05in, "s"] \arrow[r, shift right=0.05in, "id"'] 
     & \Omega S^1_\E \arrow[r] & 1_\E
    \end{tikzcd}  
  \end{center}
   The fact that the diagram is an inclusion implies that the map $U \to 1$ is an inclusion. However, 
   the fact that there exists a map $1 \to U$ implies that 
   $U$ is the maximal subobject of $1$, which means $U = 1$. Hence, $\mathbb{N}^{can}_\E$ is a Freyd natural number object.
  \end{proof}
  
    Thus, we have proven that $\mathbb{N}^{can}_\E$ is a Peano natural number object and Freyd natural number object in $\E$. 
    The goal of the next section is to show that it is also a Lawvere natural number object in $\E$.

\section{Finite Cardinals and Contractibility in \texorpdfstring{$(\infty,1)$}{(oo,1)}-Categories}
 In the previous section we proved that every $(\infty,1)$-category that satisfies the conditions of \cref{Subsec Notation}
 has a Peano and Freyd natural number object $\mathbb{N}^{can}_\E$. 
 What remains to prove is that $\mathbb{N}^{can}_\E$ is also a Lawvere natural number object.
 Unfortunately, unlike the classical case proving this is quite challenging (as we have discussed in more detail in  \cref{Subsec Motivation}). Hence, we need to introduce and prove certain logical constructions in $(\infty,1)$-categories.

 \subsection{Finite Cardinals} \label{Subsec Finite Cardinals}
 Finite cardinals have been studied extensively in the context of elementary toposes \cite[Subsection A2.5]{johnstone2002elephanti}, \cite[Subsection D5.2]{johnstone2002elephantsii}. We want to generalize certain aspects of finite cardinals to $(\infty,1)$-categories and their subcategory of $0$-truncated objects. So, for this section $\E$ is a $(\infty,1)$-category satisfying \cref{Subsec Notation}, $\tau_0\E$ its underlying elementary and $\mathbb{N}$ a Peano natural number object in $\E$ (which exists by \cref{Lemma EHT Peano NNO}).
 
 \begin{defone} \label{Def Generic Finite Cardinal}
 	Let $\mathbb{N}_1$ be the object defined in the following pullback
 	\begin{center}
 		\pbsq{\mathbb{N}_1}{1}{\mathbb{N} \times \mathbb{N}}{\mathbb{N}}{}{inc(n \leq m)}{o}{\dot -}
 	\end{center}
 	where $\dot -: \mathbb{N} \times \mathbb{N} \to \mathbb{N}$ is the truncated subtraction as described on \cite[Example A2.5.4]{johnstone2002elephanti}.
 	This is known as the {\it generic finite cardinal}
 	in $\mathbb{N}$. 
 \end{defone}
  
  \begin{remone} \label{Rem N One in Spaces}
 	We should think of the truncated subtraction $\dot - : \mathbb{N} \times \mathbb{N} \to \mathbb{N}$ as the map that takes a tuple $(n,m)$ to 
 	the maximum of $0$ and $n-m$. Thus $n \dot - m = 0$ if and only if $n \leq m$ in $\mathbb{N}$.
 	This means we can depict $\mathbb{N}_1$ as 
 	\begin{center}
 		\begin{tikzcd}[row sep=0.05in, column sep=0.05in]
 			(0,0) & (0,1) & (0,2) & (0,3) & ... \\
 			& (1,1) & (1,2) & (1,3) & ... \\
 			& & (2,2) & (2,3) & ... 
 		\end{tikzcd}.
 	\end{center}
 \end{remone}
   
  For a given map $p:X \to \mathbb{N}$ we define the finite cardinal $[p]$ in $\E_{/X}$ as the pullback along the map $s\pi_2: \mathbb{N}_1 \to \mathbb{N}$. Moreover, we say $p$ is a {\it non-empty} finite cardinal if $p$ factors through $s$. This is equivalent to saying that the pullback of $p$ along $s$ is just the identity.
   	In the classical setting of spaces a finite cardinal corresponds to a finite set $[n] = \{ 0, ... , n-1\}$.
  For more details on the definition of a finite cardinal see \cite[Page 114]{johnstone2002elephanti}. 

 Using the definition of finite cardinalities we can prove $o^*\mathbb{N}_1 \cong 0$ and $(sp)^*\mathbb{N}_1 \cong 1 \coprod [p]$
 (by \cite[Lemma A.2.5.14]{johnstone2002elephanti}) and that the map $\mathbb{N}_1 \to \mathbb{N} \times \mathbb{N}$ is a linear order on $\mathbb{N}$.
 	(\cite[Proposition A.2.5.10,Proposition A.2.5.12,Proposition A.2.5.13]{johnstone2002elephanti}.
 	Indeed, all these results only depend on the Peano natural number structure and hence directly hold in $\E$.
 
 \begin{remone} \label{Rem Ordering Succ Card}
 	Every finite cardinal $[p]$ has a linear ordering (which is induced by the linear ordering on $\mathbb{N}$).
 	Henceforth when we refer to the finite cardinal $[sp] \cong 1 \coprod [p]$, the object
 	$\iota_1 : 1 \to 1 \coprod [p]$ will be the {\it minimum object} in $[sp]$ and correspond to the map 
 	$o: 1 \to [sp]$.
 \end{remone}
 
 \begin{lemone} \label{Lemma Induction on Finites}
 	\cite[Lemma D.5.2.1]{johnstone2002elephantsii} 
 	Let $P$ be a property of objects which is expressible in the internal language of the elementary topos $\tau_0 \E$, and suppose 
 	\begin{itemize}
 		\item[(i)] The initial object $0$ ($1$) satisfies $P$.
 		\item[(ii)] Whenever an object $A$ satisfies $P$, so does $1 \coprod A$. 
 	\end{itemize}
 	Then every (non-empty) finite cardinal satisfies $P$.
 \end{lemone}
 
 We end this section with some basic, but important, observations about finite cardinals.
 
 \begin{lemone}
 	\cite[Lemma D5.2.9]{johnstone2002elephantsii}
 	Let $[p]$ be a finite cardinal. Then the map $[p] \dot - (-) : [sp] \to [sp]$ is an automorphism of $[sp]$ 
 	such that $p \dot - o: 1 \to [sp]$ is the maximal element in $[sp]$.
 \end{lemone}
 
 	Let $[p]$ be a finite cardinal. Then we define $max([sp]): 1 \to [sp]$ as $ p \dot - o : 1 \to [sp]$.
 	The map $max([sp])$ gives us the maximum element in $[sp]$,

 \begin{lemone} \label{Lemma Max p}
 	Let $[p]$ be a finite cardinal. There exists a map $inc_{[p]}: [p] \to [sp]$ such that the following is a pushout square.
 	\begin{center}
 		\begin{tikzcd}[row sep=0.3in, column sep=0.3in]
 			\emptyset \arrow[d] \arrow[r] & 1 \arrow[d, "max( \leb sp \reb )"] \\
 			\leb p \reb \arrow[r, "inc_{[p]}"'] & \leb sp \reb \arrow[ul, phantom, "\ulcorner", very near start]
 		\end{tikzcd}
 	\end{center}
 \end{lemone}
 
 \begin{proof}
 	According to \cref{Lemma Induction on Finites} we can use an inductive argument.
 	The case $[p] = 1$ is clear as we have the map $\iota_1 : 1 \to 1 \coprod 1$. 
 	Let us assume we have a map $inc_{[p]}:[p] \to 1 \coprod [p] $. 
 	Then we define $inc_{[sp]}:[sp] \to 1 \coprod [sp] $ by $inc_{[sp]} = id_1 \coprod inc_{[p]}$,
 	using the fact that $[sp] \cong 1 \coprod [p]$ and the linear ordering on $1 \coprod [p]$ described in \cref{Rem Ordering Succ Card}.
 \end{proof}
  
\subsection{The Internal Object of Contractibility} \label{Subsec The Internal Object of Contractibility}
In this subsection we will build an object that internally determines when an object is contractible. 
This involves defining the map $isContr : \O_\E \to \O_\E$, where $\O_\E$ was defined in \cref{Subsec Notation}.
We will use $isContr$ on the internal object of finite partial maps to construct the desired subobject of $\mathbb{N}$.
The definition we give here is an adaption of a definition of Shulman \cite{shulman2015elegantunivalence} in the context of model categories,
however, the results are proven independently.

\begin{defone} \label{Def IsContr}
 Define the functor of right fibrations $isContr: \O_\E \to \O_\E$ over $\E$ as 
	$$isContr(p:E \to B) = p_! (\pi_2)_* (\Delta: E \to E \times_B E).$$
\end{defone}

\begin{remone} \label{Rem IsContr Basechange Stable}
	Functoriality of $isContr$ follows from the fact that for any map $g: A \to B$, we have an equivalence 
	$$g^*(isContr(p:E \to B)) \simeq isContr(g^*p: E \to A)$$
	which immediately follows from the fact that $g^*$ commutes with  $(\pi_2)_*$ and $p_!$
\end{remone}

 The map of right fibrations restricts to a map of spaces, which deserves its own notation.
	Fix an object $B$ in $\C$, then by taking fibers we get a map of spaces
	$$isContr_B: (\C_{/B})^{core} \to (\C_{/B})^{core}.$$
	Let us give a more detailed explanation of the definition. For a given map $p:E \to B$ we have following diagram.
	\begin{center}
		\begin{tikzcd}[row sep=0.25in, column sep=0.25in]
			E \arrow[d, "\Delta"] \arrow[r, dashed] & (\pi_2)_* E \arrow[d, "(\pi_2)_* \Delta", dashed] \arrow[r, dashed] & 
			p_! (\pi_2)_* E \arrow[d, "p (\pi_2)_* \Delta", dashed] \\
			E \underset{B}{\times} E \arrow[r, "\pi_2"] & E \arrow[r, "p"] & B
		\end{tikzcd}
	\end{center}
	where $(\pi_2)_* : \E_{/ E \times_B E} \to \E_{/ E}$ is the pushforward functor (right adjoint of the pullback $(\pi_2)^*$) 
	and $p_!: \E_{/E} \to \E_{/B}$ is the left adjoint to the 
	pullback $p^*$.
The map $isContr_B$ has following important properties.

\begin{lemone} \label{Lemma IsContr giving Contractibility}
	For an object $p:E \to B$ in $\E_{/B}$ the following are equivalent.
	\begin{enumerate}
		\item $p$ is equivalent to the final object in $\E_{/B}$.
		\item $isContr_B(E)$ is equivalent to the final object in $\E_{/B}$.
		\item The space $\Map_{/B}(B, isContr_B(E))$ is non-empty.
	\end{enumerate}
\end{lemone}

\begin{proof}
    The proofs of {\it (1) $\Rightarrow$ (2) $\Rightarrow$ (3)} are immediate, hence we focus on {\it (3) $\Rightarrow$ (1)}.
	Let $H: B \to p_! (\pi_2)_* \Delta$. Then this means we have following diagram.
	\begin{center}
		\begin{tikzcd}[row sep=0.25in, column sep=0.25in]
			& (\pi_2)_* \Delta \arrow[d] \\
			& E \arrow[d, "p"]\\
			B \arrow[uur, dashed, "H"] \arrow[ur, "s", dashed] \arrow[r, "id_B"] & B 
		\end{tikzcd}
	\end{center}
	where $s:B \to E$ is simply the composite map, and thus a section of $p: E \to B$. 
	This means that $H$ is also a lift of $s$ i.e. an element in $\Map_{/E}(B, (\pi_2)_* \Delta)$. Now we have following pullback square 
	\begin{center}
		\pbsq{E}{B}{E \underset{B}{\times} E}{E}{p}{( sp \comma 1_E )}{s}{\pi_2} . 
	\end{center}
	This gives us an adjunction $\Map_{/E \times_B E}(E, \Delta) \simeq \Map_{/E}(B, (\pi_2)_* \Delta)$. By the adjunction we get a map 
	$\tilde{H} : E \to E$ that fits into following commutative diagram
	\begin{center}
		\begin{tikzcd}[row sep=0.25in, column sep=0.25in]
			E \arrow[rr, "\tilde{H}"] \arrow[dr, "(sp \comma 1_E)"'] & & E \arrow[dl, "( 1_E \comma 1_E )"] \\
			& E \underset{B}{\times} E & 
		\end{tikzcd}
	\end{center}
	This implies that $sp$ is equivalent to $1_E$, which means that $p: E \to B$ is equivalent to the final object.
\end{proof}

\begin{lemone} \label{Lemma IsContr is Mono}
	The object $isContr_B(E)$ is $(-1)$-truncated.
\end{lemone}

\begin{proof}
	First of all $\Map_{/B}(B,isContr_B(E))$ is non-empty if and only if $p: E \to B$ is equivalent to $id_B: B \to B$, 
	in which case $isContr_B(E)$ is the final object, which implies that 
	$$\Map_{/B}(B,isContr_B(E))=* .$$ 
	Now, for any other map $g: A \to B$ we have following equivalences
	$$\Map_{/B}(g, isContr_B(E)) \simeq \Map_{/A}(id_A, g^*(isContr_B(E))) \simeq \Map_{/A}(id_A, isContr_A(g^*E))$$
	where the first equivalence follows from the adjunction and the second from the functoriality. 
	But we have already proven that the space $\Map_{/A}(id_A, isContr_B(g^*E))$ is empty or contractible and hence 
	we are done.
\end{proof}

\section{Peano Implies Lawvere} \label{Sec Peano NNO Eq Lawvere NNO Eq Freyd NNO}
   In this section we prove that every Peano natural number object in an $(\infty,1)$-category $\E$ that satisfies the conditions of 
   \cref{Subsec Notation} is a Lawvere natural number object and then use it to prove that all notions of natural number object coincide and exists.
   
   \subsection{Outline and Implication}
   As explained in \cref{Subsec Motivation} the proof that Peano natural number objects are Lawvere natural numbers in $\E$ is quite intricate. 
   Hence, we will give a breakdown of the proof, some implications, an intuition and then complete the proof in the next sections.
   	As we are only working internal to $\E$ we will denote the natural number object in $\tau_0 \E$ by $\mathbb{N}$ 
   	to simplify notation.
 
   \begin{theone} \label{The Freyd and Peano gives Lawvere}
   	Let $\E$ be a $(\infty,1)$-category satisfying the conditions of \cref{Subsec Notation} and $(\mathbb{N},o,s)$ be a Peano natural number object. Then it is a Lawvere natural number object.
   \end{theone}

   \begin{proof}
   	Let $(X,b:1 \to X, u:X \to X)$ be a triple in $\E$.
   	By \cref{Rem NNO Space} we need to prove that $\Ind(X,b,u)$ is contractible. In order to prove this we will construct a space $\Part(X,b,u)$ (\cref{Def PartXbu}), prove it is contractible (\cref{The Part Contractible}) and finally prove $\Ind(X,b,u)$ is the retract of $\Part(X,b,u)$ (\cref{Prop PartialTotal id}).
   \end{proof}
   
    	The idea of constructing a retract of the space of partial maps to study maps out of natural number objects is motivated by work of Shulman in the context of homotopy type theory \cite{shulman2020gitnno}, however, the proofs are original and are different from the analogous result in homotopy type theory.
    
   \begin{theone}
   	Let $\E$ be a $(\infty,1)$-category satisfying the conditions in \cref{Subsec Notation}, along with the triple $(\mathbb{N},o: 1 \to \mathbb{N},s: \mathbb{N} \to \mathbb{N})$.
   	Then the following are equivalent:
   	\begin{enumerate}
   		\item $(\mathbb{N},o,s)$ is a Lawvere natural number object in $\E$. 
   		\item $(\mathbb{N},o,s)$ is a Freyd natural number object in $\E$.
   		\item $(\mathbb{N},o,s)$ is a Peano natural number object in $\E$.
   	\end{enumerate}
   	Thus we can simply refer to such an object as a {\it natural number object} in $\E$.
   	Moreover, $\E$ always has a natural number object.
   \end{theone}
   
   \begin{proof}
   	First, we prove that a natural number object exists and then we prove it is unique.
   	By \cref{Lemma EHT Peano NNO} $\E$ has a Peano natural number object $\mathbb{N}^{can}$. By \cref{Prop EHT Freyd NNO} $\E$,  $\mathbb{N}^{can}$ is a Freyd natural number object as well.
   	Finally, by \cref{The Freyd and Peano gives Lawvere}, $\mathbb{N}^{can}$ is a Lawvere natural number object as well. 
   	Hence, $\E$ has a triple $(\mathbb{N}^{can},o,s)$ that is a Peano, Freyd and Lawvere natural number object. 
   	
   	Now, we prove if a triple $(\mathbb{N}, o,s)$ is one type of natural number object, then it is also the other two.
   	Let $\mathbb{N}$ be a Peano, Freyd or Lawvere natural number object in $\E$. Then it is the same type of natural number object in the underlying elementary topos $\tau_0\E$. Hence, by the uniqueness result in \cref{lemma:nno unique}, it must be isomorphic to $\mathbb{N}^{can}$.
   \end{proof}

   The next subsections focus on completing the steps in the proof of \cref{The Freyd and Peano gives Lawvere}. 
   Before we start let us give an intuition for the proof. Let $(X,b,u)$ be a space $X$ along with a chosen point $b$ in $X$ and a map $u: X \to X$.
   Then we can define a map from the natural numbers $f: \mathbb{N} \to X$ as $f(n) = u^n(b)$ (where $u^0(b) = b$). 
   This can be depicted as a sequence of elements

$$
    \begin{tikzcd}[row sep=0.1in, column sep=0.1in]
      b & u(b) & u^2(b) & u^3(b) & \cdots
    \end{tikzcd}.
$$   

    The goal is to show that the space of maps $\mathbb{N} \to X$ (satisfying this condition) is contractible using 
    finite methods. However, the problem is that the set $\mathbb{N}$ 
   is infinite. We thus want to replace $\mathbb{N}$ with objects that we can study using induction.
   
   We can extend any such map $f: \mathbb{N} \to X$ to a family of compatible maps $f_p: [p] \to X$, 
   where by compatibility we mean that when we restrict $f_{n}$ to the domain $[n-1]$ we get $f_{n-1}$
   \begin{equation} \label{eq:point in part}
    \begin{tikzcd}[row sep=0in, column sep=0.03in]
     f_0 & f_1 & f_2 & f_3 & \cdots \\
     b & b & b & b & \cdots \\
      & u(b) & u(b) & u(b) & \cdots \\
      & & u^2(b) & u^2(b) & \cdots \\
      & & & u^3(b) & \cdots 
    \end{tikzcd}.
   \end{equation}
   Notice that we can recover the original maps $f$ simply by taking the diagonal of \cref{eq:point in part}.
   
   But the maps $f_p$ all have a finite domain. Thus we can study the collection of maps $\{ f_n \}_{n \in \mathbb{N}}$, using induction.
   In particular, we prove the desired contractibility result by proving that the space of partial maps is contractible.
   This suggests that we should first study collection of compatible finite maps in $\E$  (\cref{Subsec Space of Partial Maps}).
   Then we show that we can extend a map $\mathbb{N} \to X$ to a compatible family of finite maps (as in \cref{eq:point in part}) and restrict 
   it back to a map $\mathbb{N} \to X$ 
   using the diagonal. This will then show that the space of maps $\mathbb{N} \to X$ is also contractible 
   (\cref{Subsec Total vs. Partial Maps}).

\subsection{Space of Partial Maps} \label{Subsec Space of Partial Maps}
  In this subsection we define the space of partial maps and establish some important facts about it, in particular proving it is contractible (\cref{The Part Contractible}). 
    Let $[p]$ be a finite non-empty cardinal in $\E$. Moreover, let $(X,b:1 \to X,u:X \to X)$ be a triple in $\E$. Define the {\it space of finite partial maps}, $Part([p], (X,b,u))$, as the space of maps
    $f: [p] \to X$ that make the following diagram commute
    \begin{center}
     \begin{tikzcd}[row sep=0.05in, column sep=0.25in]
       & \leb p \reb \arrow[r, "\iota_2"] \arrow[dd, "finc_{p}", dashed] & 1 \coprod \leb p \reb  \arrow[dd, "f", dashed] \\
       1 \arrow[ur, "o"] \arrow[dr, "b"'] & & \\
       & X \arrow[r, "u"] & X
     \end{tikzcd}
    \end{center}
    where $inc_p$ is defined in \cref{Lemma Max p}.
   
   In the case of spaces, $Part([p], (X,b,u))$ is the space of maps 
   $f: [p] \to X$ such that $f(0) = b$ and $f(i) = u^i(b)$ where $1 \leq i \leq p-1$.
   Notice this determines the map $f$ uniquely.
  
 \begin{defone} \label{Def PartXbu}
  For a triple $(X,b:1 \to X,u: X \to X)$ in $\E$ define 
  $$Part(X,b,u) = Part(\pi_2:\mathbb{N}_1 \to \mathbb{N},(\pi_2: X \times \mathbb{N} \to \mathbb{N},(b, id_\mathbb{N}): \mathbb{N} \to X \times \mathbb{N},
  u \times id_{\mathbb{N}}: X \times \mathbb{N} \to X \times \mathbb{N}))$$ 
  in $\E_{/ \mathbb{N}}$ and call it the {\it space of partial maps}.
 \end{defone}
 
   The fact that $0 \ \dot - \ n = 0$ implies that the map $\pi_2 : \mathbb{N}_1 \to \mathbb{N}$ 
   has a section $o: \mathbb{N} \to \mathbb{N}_1$.
   This is precisely the minimum map of a finite cardinal $o:\mathbb{N} \to \mathbb{N}_1$ in $\E_{/\mathbb{N}}$ 
   that is the induced by the linear order on $\mathbb{N}$.
 
 \begin{exone} \label{Ex Partial Maps Spaces}
  In \cref{Rem N One in Spaces} we gave a description of $\mathbb{N}_1$ in spaces.
  Using that we realize that a point in $\Part(X,b,u)$ is given by \cref{eq:point in part}.
 \end{exone}

   We will now define an internal version of the space $\Part([p], (X,b,u))$ and discuss its similarities with the space $Part([p], (X,b,u))$.
    Let $[p]$ be a finite cardinal and $(X,b,u)$ as before. Let $\uPart( [p], (X,b,u))$, the {\it internal object of finite partial maps},
    be defined as the limit of the following diagram in $\E$
    \begin{equation} \label{eq:internal Lawvere limit}
     \begin{tikzcd}[row sep=0.3in, column sep=0.7in]
       1 \arrow[r, "b"] & X & X^{[sp]} \arrow[l, "o^*"'] \arrow[r, "(u_*inc_p^* \comma {\iota_2}^*)"]  & X^{[p]} \times X^{[p]} & 
       X^{[sp]} \arrow[l, "( inc_p^* \comma inc_p^*)"'] 
     \end{tikzcd}.
    \end{equation}  
   
   \begin{lemone} \label{Lemma Part Rep Condition}
    We have an equivalence of spaces
    $$ \Map(C, \uPart([p],(X,b,u))) \simeq \Part([p], (X^C,b^C,u^C)).$$
   \end{lemone}
   
   \begin{proof}
    We have following equivalence of spaces
    $$\Map(C,X^{[p]}) \simeq \Map(C \times [p],X) \simeq \Map([p],X^C) \simeq \Map(1,(X^C)^{[p]})$$
    Thus it suffices to check for the case $C=1$.
    But both are defined in terms of similar limit diagrams (\cref{eq:limit Lawvere ext} and \cref{eq:internal Lawvere limit}) thus the result follows immediately from the fact that $\Map(1, -)$ commutes with limits.  
   \end{proof}
   
   Similarly, we can internalize the space of partial maps with the same property.
    For $(X,b,u)$ we define
    $$ \uPart(X,b,u) =  \uPart(\pi_2:\mathbb{N}_1 \to \mathbb{N},(\pi_2: X \times \mathbb{N} \to 
    \mathbb{N},(b, id_\mathbb{N}): \mathbb{N} \to X \times \mathbb{N},
    u \times id_{\mathbb{N}}: X \times \mathbb{N} \to X \times \mathbb{N}))$$ 
    in $\E_{/ \mathbb{N}}$ and call it the {\it internal object of partial maps}.

   \begin{corone} \label{Cor Two Part over N}
    We have an equivalence of spaces $ \Map_{/ \mathbb{N}}(\mathbb{N} , \uPart(X,b,u)) \simeq \Part(X,b,u)$.
   \end{corone}
   
   The same way that $\mathbb{N}_1 \to \mathbb{N}$ is the generic finite cardinal, the object $ \uPart(X,b,u) \to \mathbb{N}$ 
   is the generic finite partial map. In fact we have following result.
   
   \begin{propone} \label{Prop Part generic finite partial map}
    Let $p$ be a finite cardinal. Then we have following pullback diagram
    \begin{center}
     \pbsq{\uPart([p], (X,b,u))}{\uPart(X,b,u)}{1}{\mathbb{N}}{}{}{}{p}
    \end{center}
   \end{propone}

   \begin{proof}
    This follows immediately from the fact that the fiber of an internal mapping object is by definition the mapping object of the fibers.
    Concretely, in our situation we have the pullback square 
    \begin{center}
     \begin{tikzcd}[row sep=0.25in, column sep=0.25in]
       X^{[p]} \arrow[r] \arrow[d] \arrow[dr, phantom, "\ulcorner", very near start] & \dumbmacro \arrow[d] \\
       1 \arrow[r, "p"] & \mathbb{N}
     \end{tikzcd}
    \end{center}
    and the limit diagram is preserved by pullbacks as limits commute.
   \end{proof}
 
 We want to now prove that the space $\Part(X,b,u)$ is contractible. Ideally we would have used \cref{Lemma Induction on Finites}. However, that can be used only for properties that can be expressed internal to the underlying elementary topos $\tau_0\E$. What we would need is a higher categorical version of induction, however, that has not been proven yet. Hence, we will use the results from \cref{Subsec The Internal Object of Contractibility}.
  
 \begin{theone} \label{The Part Contractible}
 	The space of finite partial maps $\Part(X,b,u)$ is contractible.
 \end{theone}
 
 \begin{proof}
 	We will show that $\uPart(X,b,u)$ is contractible. The result then follows from \cref{Cor Two Part over N}.
 	By \cref{Lemma IsContr giving Contractibility}, it suffices to prove that 
 	$isContr_\mathbb{N}(\uPart(X,b,u)) \to \mathbb{N}$
 	is the final object. 
 	By \cref{Lemma IsContr is Mono}, $isContr_\mathbb{N}(\uPart(X,b,u))$ is a subobject of $\mathbb{N}$.
 	But we know that $\mathbb{N}$ is a Peano natural number object. Thus it suffices to prove that 
 	$isContr_\mathbb{N}(\uPart(X,b,u))$ is closed under the maps $o$ and $s$.
 	\par 
 	Now, the fact that the fiber of $\uPart(X,b,u)$ over the point $p: 1 \to \mathbb{N}$ is $\uPart([p], (X,b,u))$
 	(\cref{Prop Part generic finite partial map}) and that $isContr$ commutes with basechange (\cref{Rem IsContr Basechange Stable}) implies that 
 	the fiber of $isContr_{\mathbb{N}}( \uPart(X,b,u))$ over $p$ is just $isContr(\uPart([p], (X,b,u)))$. 
 	\par 
 	Thus in order to show that $isContr_\mathbb{N}(\uPart(X,b,u))$ is closed under the maps $(o,s)$ we have to prove that
 	$\uPart(1, (X,b,u))$ is the final object and that 
 	if $\uPart(A, (X,b,u))$ is the final object then $\uPart(1 \coprod A , (X,b,u))$ is the final object. 
 	However, by \cref{Lemma Part Rep Condition}, $\uPart(A, (X,b,u))$ is the final object if and only if $\Part(A, (X,b,u))$ 
 	is contractible, for all $A$. Thus we can translate those two conditions into proving that $\Part(1, (X,b,u))$ is a contractible space, 
 	which holds trivially, and the statement that if $\Part(A, (X,b,u))$ is contractible then $\Part(1 \coprod A , (X,b,u))$ is also a contractible space.  
 	
 	Using a similar argument we know that $\Part(A,(X,ub,u))$ is contractible. 
 	So fix a point $f: 1 \coprod A \to X$ in $\Part(A,(X,ub,u))$. That means we have the diagram 
 	\begin{equation} \label{eq:the diagram}
 		\begin{tikzcd}[row sep=0.1in, column sep=0.25in]
 			&[0.2in] A \arrow[r, "\iota_2"] \arrow[dd, "f(inc_A)" description, dashed] & 1 \coprod A  \arrow[dd, "f", dashed] \\
 			1 \arrow[ur, "o"] \arrow[dr, "ub"'] & & \\
 			& X \arrow[r, "u"] & X
 		\end{tikzcd}
 	\end{equation}
 	We will prove that $\Part(1 \coprod A, (X,b,u))$ is contractible as well.
 	From \cref{eq:the diagram} we get the diagram 
 	\begin{center}
 		\begin{tikzcd}[row sep=0.1in, column sep=0.25in]
 			&[0.2in] 1 \coprod A \arrow[r, "\iota_2"] \arrow[dd, "b \coprod f(inc_A)" description, dashed] & 1 \coprod (1 \coprod A)  \arrow[dd, "b \coprod f", dashed] \\
 			1 \arrow[ur, "\iota_1"] \arrow[dr, "b"'] & & \\
 			& X \arrow[r, "u"] & X
 		\end{tikzcd}.
 	\end{center}
 	In order for $b \coprod f$ to be an element in $\Part(1 \coprod A,(X,b,u))$ we have to show that 
 	$(b \coprod f) \circ inc_{1 \coprod A} = b \coprod (f \circ inc_A)$.
 	However, from \cref{Lemma Max p} we know that $inc_{1 \coprod A} = id_1 \coprod inc_A$. Thus we have 
 	$$(b \coprod f) \circ inc_{1 \coprod A} = (b \coprod f) \circ (id_1 \coprod inc_A) = b \coprod (f \circ inc_A) .$$ 
 	This proves that the space non-empty. We now have to show it is contractible.
 	We have a map
 	$$ (\iota_2)^*: \Part(1 \coprod A, (X,b,u)) \to \Part(A, (X,ub,u))$$
 	that can be depicted as the following diagram
 	\begin{center}
 		\begin{tikzcd}[row sep=0.1in, column sep=0.25in]
 			&[0.2in] A \arrow[r, "\iota_2"] \arrow[dd, "\iota_2"] & 1 \coprod A \arrow[dd, "\iota_2"] \\
 			\\
 			& 1 \coprod A \arrow[r, "\iota_2"] \arrow[dd, "f(inc_{1 \coprod A})" description, dashed] & 1 \coprod (1 \coprod A)  \arrow[dd, "f", dashed] \\
 			1 \arrow[ur, "\iota_1"] \arrow[dr, "b"'] & & \\ 
 			& X \arrow[r, "u"] & X
 		\end{tikzcd}.
 	\end{center}
 	The map $(\iota_2)^*$ gives us following pullback square of spaces 
 	\begin{center}
 		\pbsq{F_g}{Part(1 \coprod A, (X,b,u))}{*}{\Part(A, (X,ub,u))}{\simeq}{\simeq}{(\iota_2)^*}{g}
 	\end{center}
 	where $g: 1 \coprod A \to X$ is a choice of element in $\Part(A, (X,ub,u))$. The fiber $F_g$ is the subspace of the form
 	$$ \{f \in Part(1 \coprod A, (X,b,u)): f \iota_2 = g \} . $$
 	This space is contractible as $f$ is uniquely determined by $g$ and the initial value $b$. However, by induction assumption $\Part(A, (X,ub,u)$ 
 	is contractible and	this implies that $\Part(1 \coprod A, (X,b,u))$ is contractible as well.
 \end{proof}
 
 Notice as part of the proof we also showed following result.
 
 \begin{corone} \label{Cor Part p Xbu Contr}
 	For any finite cardinal $[p]$, the space $\Part([p], (X,b,u))$ is contractible.
 \end{corone}
 
 Before we move on, let us do a careful analysis of the space $\Part(X,b,u)$. 
 
 \begin{remone} \label{Rem Inc N Leq M}
 	Let $inc(n \leq m): \mathbb{N}_1 \to \mathbb{N} \times \mathbb{N}$ be the standard inclusion (see \cref{Def Generic Finite Cardinal}).
 	Similarly, we take $inc(n \leq sm): \mathbb{N} \coprod \mathbb{N}_1 \to \mathbb{N} \times \mathbb{N}$, to be the standard inclusion
 	that takes $n \in \mathbb{N}$ to $(0,n)$ and $(n,m) \in \mathbb{N}_1$ to $(sn,m)$. This identifies $\mathbb{N} \coprod \mathbb{N}_1$ as the subobject
 	of $\mathbb{N} \times \mathbb{N}$ consisting of elements $(n,m)$ such that $n \leq sm$.
 	Moreover, the map $inc(n \leq m): \mathbb{N}_1 \to \mathbb{N} \times \mathbb{N}$ factors through the subobject $\mathbb{N} \coprod \mathbb{N}_1$.
 	The resulting map is exactly $inc_{\mathbb{N}_1}: \mathbb{N}_1 \to \mathbb{N} \coprod \mathbb{N}_1$.
 \end{remone}
 
 The map $\iota_2: \mathbb{N}_1 \to \mathbb{N} \coprod \mathbb{N}_1$ corresponds to the restriction 
 of the map $s \times id_{\mathbb{N}} : \mathbb{N} \times \mathbb{N} \to \mathbb{N} \times \mathbb{N}$.
 Hence, we will use the notation $s \times id_{\mathbb{N}} : \mathbb{N}_1 \to \mathbb{N} \coprod \mathbb{N}_1$ instead of $\iota_2$.
 Also, we use the notation $inc(n \leq m): \mathbb{N}_1 \to \mathbb{N} \coprod \mathbb{N}_1$ for the map $inc_{\mathbb{N}_1}$.
 We will use this notation to better understand the space $\Part(X,b,u)$.
 
 	\cref{The Part Contractible} tells us that there is a map 
 	$$ (pm, id_{\mathbb{N}} \coprod \pi_2): \mathbb{N} \coprod \mathbb{N}_1 \to X \times \mathbb{N}$$
 	over $\mathbb{N}$ such that it uniquely fills the diagram below.
 	
 	\begin{center}
 		\begin{tikzcd}[row sep=0.25in, column sep=0.7in]
 			& \mathbb{N}_1 \arrow[rr, "s \times id_{\mathbb{N}}"] \arrow[dd, dashed, "(pm (inc(n \leq m)) \comma \pi_2)" description] \arrow[dr, "\pi_2"] & & 
 			\mathbb{N} \coprod \mathbb{N}_1 \arrow[dr, "id_{\mathbb{N}} \coprod \pi_2"] \arrow[dd, dashed, " (pm \comma id_{\mathbb{N}} \coprod \pi_2)" description] \\
 			\mathbb{N} \arrow[ur, "o"] \arrow[dr, "(b \comma id_{\mathbb{N}})"'] & & \mathbb{N} & & \mathbb{N} \\
 			& X \times \mathbb{N} \arrow[ur, "\pi_2"] \arrow[rr, "u \times id_\mathbb{N}"] & & X \times \mathbb{N} \arrow[ur, "\pi_2"]
 		\end{tikzcd}.
 	\end{center}

 	Let $[p]$ be a finite cardinal. Then we get a map 
 	$restr_{[p]}: \Part(X,b,u) \to \Part([p], (X,b,u))$
 	by pulling back the diagram in $\Part(X,b,u)$ (\cref{Def PartXbu}) along the map $p: 1 \to \mathbb{N}$.
 	Concretely we get the diagram
 	\begin{center}
 		\begin{tikzcd}[row sep=0.1in, column sep=0.3in]
 			&[0.3in] \leb p \reb \arrow[rr, "\iota_2"] \arrow[dd, dashed, "pm_{[p]} (inc_{[p]})" description]  & & 
 			1 \coprod \leb p \reb \arrow[dd, dashed, "pm_{[p]}" description] \\
 			1 \arrow[ur, "o"] \arrow[dr, "b"'] & &  & & \\
 			& X \arrow[rr, "u"] & & X
 		\end{tikzcd}
 	\end{center}
 	we call the image of $pm$ under the map $restr_{[p]}(pm) = pm_{[p]}$.
 	 	By \cref{Cor Part p Xbu Contr}, the space $\Part([p], (X,b,u))$ is contractible and so every point is equivalent to $pm_{[p]}$.
     As $pm_{[p]}$ is the pullback of $pm$ this gives us following important result that we will need for the proof of \cref{Prop Partial Maps Parametrized}.

 \begin{lemone} \label{Lemma PM restricts to PM P}
 	Let $p$ be a finite cardinal and $m \leq sp$. Then $pm(m,p) \simeq pm_{[p]}(m,p)$. 
 \end{lemone}
 
 The maps $pm_{[p]}$ satisfy following important stability property.
 
 \begin{lemone} \label{Lemma PM Restricts}
 	Let $[p]$ be finite cardinal. Then we have $pm_{1 \coprod [p]} \circ  inc_{(1 \coprod [p])} \simeq pm_{[p]}.$
 \end{lemone}
 
 \begin{proof}
 	We have following diagram
 	\begin{center}
 		\begin{tikzcd}[row sep=0.35in, column sep=0.7in]
 			&[0.3in] \leb p \reb \arrow[r, "\iota_2"] \arrow[d, "inc_{[p]}" description] & 1 \coprod \leb p \reb \arrow[d, "inc_{(1 \coprod [p])}" description] \\  
 			1  \arrow[ur, "o"] \arrow[r, "\iota_1"] \arrow[dr, "b"'] & 
 			1 \coprod \leb p \reb \arrow[r, "\iota_2"] \arrow[d, "pm_{(1 \coprod [p])}inc_{1 \coprod [p]}" description] &
 			1 \coprod ( 1 \coprod [p] ) \arrow[d, "pm_{(1 \coprod [p])}" description] \\
 			& X \arrow[r, "u"] & X
 		\end{tikzcd}.
 	\end{center}
 	The fact that it commutes implies that $pm_{(1 \coprod [p])}inc_{1 \coprod [p]}$ is in $\Part([p], (X,b,u))$. 
 	However, this space is contractible, which implies that $pm_{1 \coprod [p]} inc_{(1 \coprod [p])} \simeq pm_{[p]}$.
 \end{proof}
 
 We can use these results to prove that $pm$ satisfies following important property that we use in \cref{Subsec Total vs. Partial Maps} to prove the map $\Total$ is well-defined.
 
 \begin{propone} \label{Prop Partial Maps Parametrized}
 	We have an equivalence $pm(id \times s) \circ \Delta \simeq pm \circ \Delta $.
 \end{propone}
 
 \begin{proof}
 	In order to prove we have an equivalence we show that in the following equalizer diagram 
 	\begin{equation} \label{eq:equalizer}
 		\begin{tikzcd}[column sep=0.8in]
 			Eq \arrow[r] & \mathbb{N} \arrow[r, shift left=0.05in, "pm (id \times s) \circ \Delta"] 
 			\arrow[r, shift right=0.05in, "pm \circ \Delta "'] & X
 		\end{tikzcd}
 	\end{equation}
 	the map $Eq \to \mathbb{N}$ is the final object in $\E_{/ \mathbb{N}}$. According to \cref{Lemma IsContr giving Contractibility} 
 	the result will follow if we prove that $isContr_{\mathbb{N}}(Eq)$ is the final object in $\E_{/ \mathbb{N}}$.
 	However, we also know that $isContr_{\mathbb{N}}(Eq) \to \mathbb{N}$ is mono (\cref{Lemma IsContr is Mono}) and $\mathbb{N}$ is a 
 	Peano natural number object. 
 	This means that $isContr_{\mathbb{N}}(Eq)$ is final if it is closed under the maps $o$ and $s$.\
 	\par 
 	Before we can continue the proof, we need to better understand restrictions of this equalizer diagram by finite cardinals.
 	Let $[p]$ be a finite cardinal. The inclusion map $inc_{[p]} : [p] \to \mathbb{N}$ gives us a new equalizer diagram
 	\begin{center}
 		\begin{tikzcd}[row sep=0.3in, column sep=0.8in]
 			\leb p \reb^* Eq \arrow[d, hookrightarrow] \arrow[r] & \leb p \reb \arrow[r, shift left=0.05in, "pm (id \times s) \circ \Delta \circ inc_{[p]}"] 
 			\arrow[r, shift right=0.05in, "pm \circ \Delta  \circ inc_{[p]}"'] \arrow[d, hookrightarrow] & X \arrow[d] \\
 			Eq \arrow[r] & \mathbb{N} \arrow[r, shift left=0.05in, "pm (id \times s) \circ \Delta"] 
 			\arrow[r, shift right=0.05in, "pm \circ \Delta "'] & X
 		\end{tikzcd}.
 	\end{center}
    Making the appropriate modification to \cref{eq:equalizer}, we can observe that the two maps $pm (id \times s) \circ \Delta \circ inc_{[p]}$ and 
    $pm \circ \Delta  \circ inc_{[p]}$ are equivalent if and only if $[p]^* Eq \to [p]$ is the final object in $\E_{/ [p]}$.  
 	However, by \cref{Rem IsContr Basechange Stable} we know that $[p]^* isContr_{\mathbb{N}}(Eq) \simeq isContr_{[p]}([p]^*Eq)$.
 	Thus it reduces to showing that the subobject $[p]^* isContr_{[p]}( Eq) \to [p]$ is the final object in $\E_{/[p]}$, which implies that $isContr_{\mathbb{N}}(Eq)$ is the subobject of $\mathbb{N}$  consisting of all finite cardinals $[p]$ for which the object $isContr_{[p]}([p]^*Eq)$ is the final object. 
 	\par
 	According to the Peano axiom, in order to show that $isContr_{\mathbb{N}}(Eq)$ is the maximal subobject, we have to show that the following 
 	two statements hold. 
 	\begin{enumerate}
 		\item We have an equivalence $pm(id \times s) \circ \Delta \circ o \simeq pm \circ \Delta \circ o : 1 \to X$.
 		\item Let $[p]$ be a finite cardinal and assume we have an equivalence $pm(id \times s) \circ \Delta \circ inc_{[p]} \simeq pm \circ \Delta \circ inc_{[p]} : [p] \to X$. Then we have an equivalence $pm(id \times s) \circ \Delta \circ inc_{[sp]} \simeq pm \circ \Delta \circ inc_{[sp]} : [sp] \to X$.
 	\end{enumerate}
 	Before we prove anything let us gain a better understanding of the maps involved.
 	For an object $n:1 \to \mathbb{N}$ we have 
 	$$pm (id \times s) \circ \Delta \circ inc_{[p]}(n) =  pm (id \times s) \circ \Delta (n) = pm (id \times s) (n,n) = pm(n,sn) \simeq pm_{[sn]}(n).$$
 	The last step follows from \cref{Lemma PM restricts to PM P}.
 	Similarly, we have 
 	$pm \circ \Delta  \circ inc_{[p]}(n) \simeq pm_{[n]}(n)$.
 	By \cref{Lemma PM Restricts} we have $ pm_{[sn]}(n) \simeq pm_{[sn]} \circ inc_{[n]} (n) \simeq pm_{[n]}(n)$. 
 	\par  
 	Now we can prove numbered statements $(1)$ and $(2)$. The first one follows immediately from the previous paragraph if we use $n=o: 1 \to \mathbb{N}$.
 	We now want to prove $(2)$. Let us assume we have an equivalence 
 	$pm(id \times s) \circ \Delta \circ inc_{[p]} \simeq pm \circ \Delta \circ inc_{[p]} : [p] \to X$.
 	By \cref{Lemma Max p}, we have a pushout square
 	\begin{center}
 		\begin{tikzcd}[row sep=0.25in, column sep=0.25in]
 			\emptyset \arrow[d] \arrow[r] & 1 \arrow[d, "max(\leb sp \reb )"] \\
 			\leb p \reb \arrow[r, "inc_{[p]}"'] & \leb sp \reb \arrow[ul, phantom, "\ulcorner", very near start]
 		\end{tikzcd}.
 	\end{center}
 	We want to prove  
 	$$pm(id \times s) \circ \Delta \circ inc_{[sp]} \simeq pm \circ \Delta \circ inc_{[sp]} : [sp] \to X .$$
 	Because of the coproduct diagram it suffices to prove 
 	$$pm(id \times s) \circ \Delta \circ inc_{[sp]} \circ max([sp]) \simeq pm \circ \Delta \circ inc_{[sp]} \circ max([sp]): 1 \to X$$
 	and 
 	$$pm(id \times s) \circ \Delta \circ inc_{[sp]} \circ inc_{[p]} \simeq pm \circ \Delta \circ inc_{[sp]} \circ inc_{[p]} : [p] \to X$$
 	The first one follows from the previous paragraph, using the case $n = max([sp]) :  1 \to X$.
 	The second case follows from the induction assumption combined with the fact that $inc_{[sp]} \circ inc_{[p]} = inc_{[p]}$.
 \end{proof}
 	The content of \cref{Prop Partial Maps Parametrized} is that a point in $Part(X,b,u)$ are choices of maps $[p] \to X$ for every finite cardinal $[p]$ that are all consistent with each other. This exactly confirms our intuition from \cref{Ex Partial Maps Spaces}, where each column 
 	is the restriction of the next column.
   
 \subsection{Total vs. Partial Maps} \label{Subsec Total vs. Partial Maps}
  Our final goal it to construct maps
 $$\Ind(X,b,u) \xrightarrow{ \ \ \Partial \ \ } \Part(X,b,u) \xrightarrow{ \ \ \Total \ \ } \Ind(X,b,u)$$ 
 such that $\Total \circ \Partial$ is equivalent to the identity. Then the contractibility of $\Part(X,b,u)$ implies that 
 $\Ind(X,b,u)$ is also contractible finishing the proof of \cref{The Freyd and Peano gives Lawvere}.
 
  Here is an intuitive idea of these maps.
  The map $\Partial: \Ind(X,b,u) \to \Part(X,b,u)$ takes a map defined on $\mathbb{N}$ and restricts it to a family of 
  finite partial maps. On the other hand, the map $\Total : \Part(X,b,u) \to \Ind(X,b,u)$ takes a family of finite partial maps 
  (as depicted in \cref{eq:point in part}) to the diagonal, which gives us a full map on $\mathbb{N}$. 
  Let $\Partial : \Ind(X,b,u) \to \Part(X,b,u)$ be defined by taking $f$ to the map
  $$\Partial (f): \mathbb{N} \coprod \mathbb{N}_1 \xrightarrow{ \ \ inc(n \leq sm) \ \ } \mathbb{N}  \times \mathbb{N} 
  \xrightarrow{ \ \ f \times id \ \ } X \times \mathbb{N}. $$
  Thus $\Partial (f) = (f \times id_{\mathbb{N}}) \circ inc(n \leq sm),$ where $inc(n \leq sm): \mathbb{N} \coprod \mathbb{N}_1 \to \mathbb{N} \times \mathbb{N}$ is the inclusion map defined in \cref{Rem Inc N Leq M}.
 By definition $\Partial (f): \mathbb{N} \coprod \mathbb{N}_1 \to X \times \mathbb{N}$. However, we have to 
 prove that $\Partial (f)$ is actually a point in $Part(X,b,u)$, by showing it satisfies the right conditions.
 We have the following diagram
 
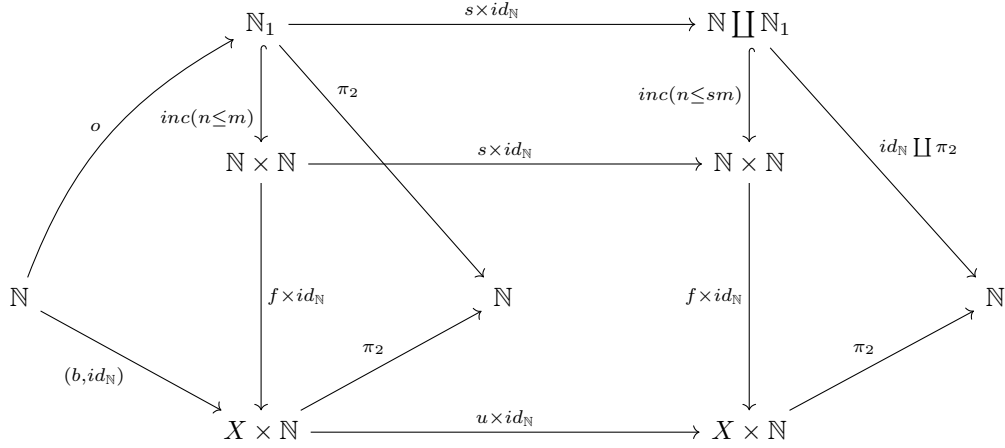
\begin{center}
    \begin{tikzcd}[row sep=0.2in, column sep=0.9in]
     & \mathbb{N}_1 \arrow[rr, "s \times id_{\mathbb{N}}"] \arrow[d, hookrightarrow, "inc(n \leq m)"' near end] \arrow[ddr, "\pi_2" near start] & & 
       \mathbb{N} \coprod \mathbb{N}_1 \arrow[ddr, "id_{\mathbb{N}} \coprod \pi_2"] \arrow[d, hookrightarrow, "inc(n \leq sm)"'] \\
     & \mathbb{N} \times \mathbb{N} \arrow[dd, "f \times id_{\mathbb{N}}"] \arrow[rr, "s \times id_{\mathbb{N}}"] & & 
    \mathbb{N} \times \mathbb{N} \arrow[dd, "f \times id_{\mathbb{N}}"'] \\  
     \mathbb{N} \arrow[uur, "o", bend left=20] \arrow[dr, "(b \comma id_{\mathbb{N}})"'] & & \mathbb{N} & & \mathbb{N} \\
     & X \times \mathbb{N} \arrow[ur, "\pi_2"] \arrow[rr, "u \times id_\mathbb{N}"] & & X \times \mathbb{N} \arrow[ur, "\pi_2"]
   \end{tikzcd}.
  \end{center}
 The bottom square commutes because $f$ is in $\Ind(X,b,u)$. The top square commutes because the vertical maps are inclusions 
 and the horizontal maps are equal. Finally, we also have to show the left side vertical map 
 ($(f \times id_{\mathbb{N}}) \circ inc(n \leq m)$) is the restriction of the right side map 
 ($(f \times id_{\mathbb{N}}) \circ inc(n \leq sm)$) 
 along $inc_{\mathbb{N}_1}$. 
 However, according to \cref{Rem Inc N Leq M}, we have $ inc(n \leq sm) \circ inc_{\mathbb{N}_1} = inc_{\mathbb{N}_1} = inc(n \leq m).$
 This immediately implies that $(f \times id_{\mathbb{N}}) \circ inc(n \leq sm) \circ inc_{\mathbb{N}_1} = (f \times id_{\mathbb{N}}) \circ inc(n \leq m)$ and proves that $\Partial (f)$ is actually in $\Part(X,b,u)$.
 
 Next we define the assignment that takes a family of partial maps to a total map.
  Let $\Total: \Part(X,b,u) \to \Ind(X,b,u)$ be defined by taking a map $(g, id_{\mathbb{N}} \coprod \pi_2)$ to the map
  $$\Total((g, id_{\mathbb{N}} \coprod \pi_2)) = \mathbb{N} \xrightarrow{ \ \ \Delta \ \ } \mathbb{N} \coprod \mathbb{N}_1 
  \xrightarrow{ \ \ (g, id_{\mathbb{N}} \coprod \pi_2) \ \ }  X \times \mathbb{N} \xrightarrow{ \ \ \pi_1 \ \ } X$$
  Concretely, it is the composition $\Total( (g, id_{\mathbb{N}} \coprod \pi_2)) = \pi_1 \circ (g \comma id_{\mathbb{N}} \coprod \pi_2) \circ \Delta . $
 The map is part of the following diagram
  
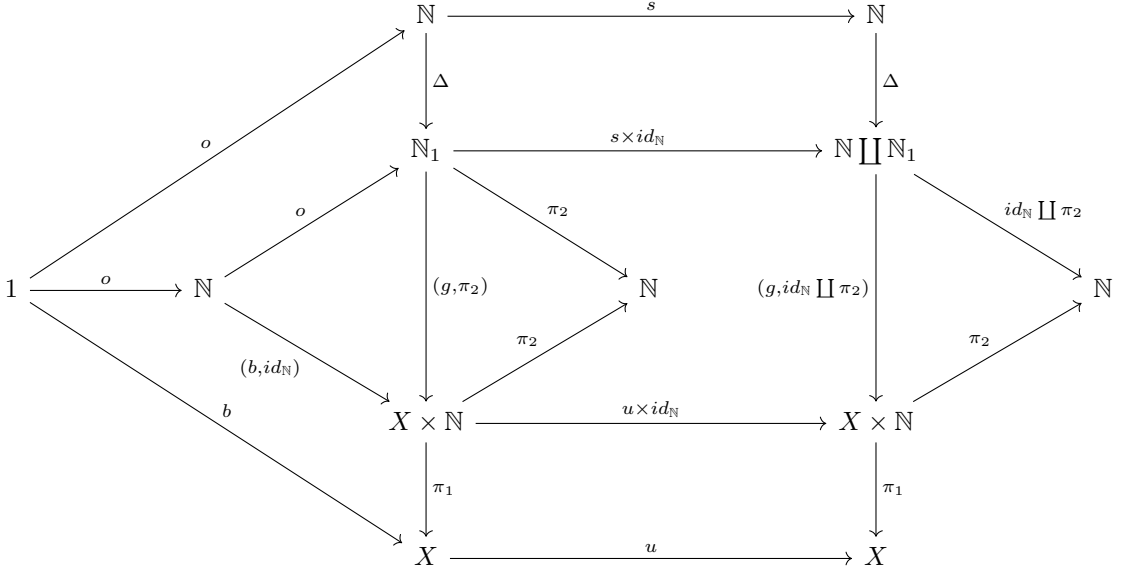
\begin{center}
   \begin{tikzcd}[row sep=0.2in, column sep=0.8in]
   & & \mathbb{N} \arrow[d, "\Delta"] \arrow[rr, "s"] & & \mathbb{N} \arrow[d, "\Delta"] \\
    & & \mathbb{N}_1 \arrow[dr, "\pi_2"]  \arrow[rr, "s \times id_{\mathbb{N}}"] \arrow[dd, "(g \comma \pi_2)"] & & 
       \mathbb{N} \coprod \mathbb{N}_1 \arrow[dr, "id_{\mathbb{N}} \coprod \pi_2"] \arrow[dd, "(g \comma id_{\mathbb{N}} \coprod \pi_2)"'] \\
   1 \arrow[ddrr, "b"] \arrow[uurr, "o"] \arrow[r, "o"] &  \mathbb{N} \arrow[ur, "o"] \arrow[dr, "(b \comma id_{\mathbb{N}})"'] & 
    & \mathbb{N} & 
    & \mathbb{N} \\
    & & X \times \mathbb{N} \arrow[ur, "\pi_2"]   \arrow[rr, "u \times id_{\mathbb{N}}"] \arrow[d, "\pi_1"] & & 
    X \times \mathbb{N} \arrow[ur, "\pi_2"] \arrow[d, "\pi_1"] \\
    & & X \arrow[rr, "u"] & & X 
   \end{tikzcd}.
  \end{center}
  We have to show that $\Total(g)$ satisfies the right conditions, which means that the large 
  rectangle commutes. Before we do that first let us analyze the inner square. 
  By assumption, we have $ (u \times id_{\mathbb{N}}) \circ (g, \pi_2) \simeq (g, id_{\mathbb{N}} \coprod \pi_2) \circ (s \times id_{\mathbb{N}}) .$
  Composing the maps we get $ (ug , \pi_2 ) \simeq (g \circ (s \times id_{\mathbb{N}}) , (id_{\mathbb{N}} \coprod \pi_2) \circ (s \times id_{\mathbb{N}})) . $ This in particular implies that $ug \simeq g( s \times id_{\mathbb{N}}).$
   We will use this equivalence to show that $\Total( (g, id_{\mathbb{N}} \coprod \pi_2))$ is in $Ind(X,b,u)$.
   \par 
  In order to get the desired result we need to show that $\pi_1 \circ (g \comma id_{\mathbb{N}} \coprod \pi_2) \circ \Delta \circ s \simeq u \circ \pi_1 \circ (g \comma \pi_2) \circ \Delta .$ First, notice that $\pi_1 \circ (g \comma id_{\mathbb{N}} \coprod \pi_2)  \simeq g$ and $\pi_1 \circ (g \comma \pi_2) \simeq g$. 
  Moreover, $\Delta \circ s = s \times s$. Thus we have to prove
  $g \circ (s \times s) \circ \Delta \simeq  u \circ g \circ \Delta .$
  However, by \cref{Prop Partial Maps Parametrized}, $g \circ (s \times s) \circ \Delta  \simeq g \circ (s \times id) \circ \Delta $, 
  as $g$ is equivalent to $pm$. 
  Thus we have to show
  $g \circ (s \times id) \circ \Delta \simeq  u \circ g \circ \Delta .$
  However, we showed in the previous paragraph that $ug \simeq g( s \times id_{\mathbb{N}})$ and so we get the desired equivalence by
  precomposing with $\Delta$. 
  This implies that $\Total( (g, id_{\mathbb{N}} \coprod \pi_2))$ is in $Ind(X,b,u)$ and finishes our argument. 
 
 \begin{propone} \label{Prop PartialTotal id}
  We have an equivalence $\Partial ( \Total (f) ) \simeq f.$
 \end{propone}
 \begin{proof}
 In order to prove it we have following commutative diagram
  \begin{center}
   \begin{tikzcd}[row sep=0.2in, column sep=0.8in]
   & & \mathbb{N} \arrow[d, "\Delta"] \arrow[rr, "s"] & & \mathbb{N} \arrow[d, "\Delta"] \\
    & & \mathbb{N}_1 \arrow[rr, "s \times id_{\mathbb{N}}"] \arrow[d, "inc(n \leq m)"] & & 
       \mathbb{N} \coprod \mathbb{N}_1 \arrow[dr, "id_{\mathbb{N}} \coprod \pi_2"] \arrow[d, "inc(n \leq sm)"'] \\
   1 \arrow[ddrr, "b"] \arrow[uurr, "o"] \arrow[r, "o"] &  \mathbb{N} \arrow[ur, "o"] \arrow[dr, "(b \comma id_{\mathbb{N}})"'] & 
   \mathbb{N} \times \mathbb{N} \arrow[d, "f \times id_{\mathbb{N}}"] &  & 
   \mathbb{N} \times \mathbb{N} \arrow[d, "f \times id_{\mathbb{N}}"] & \mathbb{N} \\
    & & X \times \mathbb{N}  \arrow[rr, "u \times id_{\mathbb{N}}"] \arrow[d, "\pi_1"] & & X \times \mathbb{N} \arrow[ur, "\pi_2"] \arrow[d, "\pi_1"] \\
    & & X \arrow[rr, "u"] & & X 
   \end{tikzcd}.
  \end{center} 
  Thus we need to show that  $\pi_1 \circ (f \times id_{\mathbb{N}}) \circ inc_{n \leq m} \circ \Delta  \simeq f . $
  In order to do that we first notice that $\pi_1 \circ (f \times id_{\mathbb{N}}) \simeq f \circ \pi_1$ which means we can also show
  $f \circ \pi_1 \circ inc_{n \leq m} \circ \Delta  \simeq f . $
  At this point the result follows immediately from the fact that $inc_{n \leq m} \circ \Delta = \Delta$ and $\pi_1 \circ \Delta = id_{\mathbb{N}}$.
  \end{proof}
   
\section{Applications of Natural Number Objects} \label{Sec Applications}
  In this last section we want to look at certain implications of the existence of natural number objects.
  The first subsection (\cref{Subsec Additional Properties of Natural Number Objects}) will focus on additional properties of the natural number object. 
  The rest focuses on the interaction between the natural number object and universes, in the context of an elementary $(\infty,1)$-topos, which we will hence review first in \cref{Subsec Not every Elementary Topos lifts to an Elementary Higher Topos}. In particular, we will study external (\cref{Subsec Infinite Colimits in an Elementary Higher Topos}) and internal (\cref{Subsec Internal Infinite Coproducts and Sequential Colimits}) infinite colimits.

\subsection{Additional Properties of Natural Number Objects} \label{Subsec Additional Properties of Natural Number Objects}
 In this short subsection we discuss some of the results that hold for natural number objects in an elementary topos and 
 directly generalize to an $(\infty,1)$-category that satisfies the conditions of \cref{Subsec Notation}. As the proofs are completely analogous we will just cite the relevant sources. 
 
 \begin{propone}\cite[Proposition A2.5.2]{johnstone2002elephanti}
  Let $(\mathbb{N}_\E, o,s)$ be a natural number object in an elementary $(\infty,1)$-topos. Then for any morphism $g: A \to B$ 
  and $h: A \times \mathbb{N}_\E \times B \to B$, there exists a unique $f: A \times \mathbb{N}_\E \to B$ 
  such that the diagram commute.
  \begin{center}
   \begin{tikzcd}[row sep=0.25in, column sep=0.25in]
    A \arrow[r, "(1_A \comma o)"] \arrow[d, "id_A"] & A \times \mathbb{N}_\E \arrow[d, "f"] & 
    \times \mathbb{N}_\E \arrow[l, "1_A \times s"'] \arrow[d, "(1_{A \times \mathbb{N}_\E} \comma f)"] \\
    A \arrow[r, "g"'] & B & A \times \mathbb{N}_\E \times B \arrow[l, "h"] 
   \end{tikzcd}
  \end{center}
 \end{propone}
 
 \begin{lemone}
 	\cite[Lemma A2.5.16]{johnstone2002elephanti}
 	Let $\mathbb{N}_\E$ be a natural number object and let $f: A \to \mathbb{N}_\E$ be a map. Then for any map $a: 1_\E \to o^*(f)$ and 
 	$t: f \to s^*(f)$ in $\E_{/\mathbb{N}_\E}$ there exists a unique section $h$ of $f$ such that $o^*(h) = a$ and $s^*(h) = th$. 
 \end{lemone}

 Finally, notice that $\mathbb{N}_\E$ has in fact addition, multiplication, exponentiation and truncated subtraction structure 
  described in \cite[Example A2.5.4]{johnstone2002elephanti} that makes $\mathbb{N}_\E$ into a semi-ring.
 Moreover, following the argument in \cite[D4.7]{johnstone2002elephantsii} it has a group completion $\mathbb{Z}_\E$ which is in fact the free group on generator.
 
  Notice that $\mathbb{Z}_\E$ has an identity element $o: 1 \to \mathbb{Z}_\E$ and an addition map $(-)+1: \mathbb{Z}_\E \to \mathbb{Z}_\E$.
  The triple $(\mathbb{Z}_\E, o, (-)+1)$ satisfies the universal property of \cref{Prop Initial Loop Algebra}, 
  which gives us following corollary.

  \begin{corone}
   We have an isomorphism $\Omega S^1_\E \cong \mathbb{Z}_\E$.
  \end{corone}
   
   Thus we have generalized the fact that the loop space of the circle is the free group on one generator to every elementary $(\infty,1)$-topos.  

 \subsection{Elementary \texorpdfstring{$(\infty,1)$}{(oo,1)}-Topos and Natural Number Objects} \label{Subsec Not every Elementary Topos lifts to an Elementary Higher Topos}
 In this subsection we introduce universes and use them to define elementary $(\infty,1)$-toposes. 
 Let $\E$ be an $(\infty,1)$-category that satisfies the conditions of \cref{Subsec Notation}.
 A map $p: \U_* \to \U$ is called a {\it universe} if the induced map of right fibrations
  $\E_{/\U} \to \O_\E$
  is an inclusion. Moreover, we say $\U$ is {\it closed} under (finite) limits, colimits and local Cartesian closure if the class of morphisms in image of the inclusion in $\O_\E$ are closed under (finite) limits,colimits and local Cartesian closure. 
 
  $\E$ is called an {\it elementary $(\infty,1)$-topos} if it has a collection of universes $p:\U_* \to \U$ closed under limits, colimits and local Cartesian closure such that the inclusions $\E_{/\U} \to \O_\E$ are jointly surjective.
  Examples include Grothendieck $(\infty,1)$-toposes \cite{lurie2009htt}, but also filter-product $(\infty,1)$-toposes \cite{rasekh2020filterquotient}. 
 
One important result about $(\infty,1)$-Grothendieck toposes is that every Grothendieck topos can be lifted to an $(\infty,1)$-Grothendieck topos
\cite[Section 11]{rezk2010toposes}. However, we can use the existence of natural number objects to show that this does not hold in the elementary setting.
 
  \begin{corone} \label{Cor FinSet not EHT}
  	If $\E$ is an elementary $(\infty,1)$-topos, then the underlying elementary topos has a natural number object. 
  	Hence, an elementary topos without natural number object (such as the category of finite sets) cannot be lifted to an elementary $(\infty,1)$-topos.
  \end{corone}

 \subsection{Infinite Colimits in an Elementary \texorpdfstring{$(\infty,1)$}{(oo,1)}-Topos} \label{Subsec Infinite Colimits in an Elementary Higher Topos}
 In general an elementary $(\infty,1)$-topos does not have infinite colimits. However, using natural number objects we can 
 find easy criteria for the existence of infinite colimits.
 First, recall that in a Grothendieck $(\infty,1)$-topos the object $\coprod_\mathbb{N} 1$ is a natural number object.
  In particular, in the $(\infty,1)$-category of spaces the set of natural numbers is a natural number object.
 
 We want to show that a similar result holds for an elementary $(\infty,1)$-topos with countable colimits.
 In fact, we want to show that this particular result implies the existence of countable colimits in every universe.
 We will need following Proposition from \cite{lurie2009htt}.
 
 \begin{propone}\cite[Proposition 4.4.2.6]{lurie2009htt}
  If $\E$ admits pushouts (pullbacks) and countable coproducts (products) then $\E$ admits colimits (limits) for all countable diagrams.
 \end{propone}
 
 As we already know that $\E$ has all pushouts and pullbacks it thus suffices to prove it has countable (co)products 
 to show that we have countable (co)limits. 
 For the next result, for a given universe $\U$ we denote the class of morphisms in the image of the inclusion by $S$. 
 Moreover, for an object $X$, we use the notation $(\E_{/X})^S$ to denote the full subcategory of $\E_{/X}$ consisting of morphisms in $S$. 
 In particular, for $X = 1$, the final object, we use $\E^S$.
 
 \begin{propone}
  Let $\U^S$ be a universe in $\E$, classifying the class of maps $S$. Then $\E^S$ is closed under countable limits and colimits 
  if and only if $\coprod_{\mathbb{N}} 1_\E$ exists in $\E^S$. 
  Moreover, in that case $\mathbb{N}_\E = \coprod_{\mathbb{N}} 1_\E$ is the natural number object in $\E$.
 \end{propone}
 
 \begin{proof}
  If $\E$ has infinite colimits, then obviously $\coprod_{\mathbb{N}} 1_\E$ exists. On the other hand, let us assume $\coprod_{\mathbb{N}} 1_\E$ exists.
  Let $F: \mathbb{N} \to \E^S$ be a fixed diagram. Then $F(n)$ is an object in $\E^S$ which corresponds to a map $1_\E \to \U^S$. Using the fact that coproduct of $1_\E$ exists we 
  thus get a map $\hat{F}: \coprod_{\mathbb{N}} 1_\E \to \U^S$. 
  Now for each $n: 1_\E \to \coprod_{\mathbb{N}} 1_\E$ we have following diagram
  \begin{equation} \label{eq:coproduct}
   \begin{tikzcd}[row sep=0.25in, column sep=0.25in]
    F(n) \arrow[dr, phantom, "\ulcorner", very near start] \arrow[d] \arrow[r] & 
    C \arrow[dr, phantom, "\ulcorner", very near start]\arrow[d] \arrow[r] & \U^S_* \arrow[d] \\
    1_\E \arrow[r, "n"] & \coprod_{\mathbb{N}} 1_\E \arrow[r, "\hat{F}"] & \U^S 
   \end{tikzcd}.
  \end{equation}
  By descent $C$ is the coproduct of $F$.
  
  Now we also show that the diagram $F: \mathbb{N} \to \E^S$ has a product. For this part we first have to recall the following.
  The map $fi: \coprod_{\mathbb{N}} 1_\E \to 1_\E$ gives us following adjunction
   \adjun{\E}{\E_{/ \coprod_{\mathbb{N}} 1_\E}}{fi^*}{fi_*}.

  Let $C$ be the coproduct of $F: \mathbb{N} \to \E^S$ given in \cref{eq:coproduct} and notice it comes with a map $C \to \coprod_\mathbb{N} 1_\E$, which means it is an object in $\E_{/ \coprod_{\mathbb{N}} 1_\E}$. We will now prove that 
  that $fi_* C$ is the product of the diagram $F: \mathbb{N} \to \E^S$. Let $Y$ be any other object. By adjunction we have the equivalences
  $$\Map_\E(Y, fi_* C) \simeq \Map_{\E_{/ \coprod_{\mathbb{N}}1_\E}}( Y \times \coprod_{\mathbb{N}} 1_\E, C) \simeq \Map_{\E_{/ \coprod_{\mathbb{N}}1_\E}}(\coprod_{\mathbb{N}} Y, C).$$
  Recall that the descent condition gives us following equivalence (see \cref{Ex Cart over Final} and notice that $\mathbb{N}$ is already a groupoid)
  \begin{center}
   \begin{tikzcd}[row sep=0.2in, column sep=0.2in]
    (\E^S)_{/\coprod_{\mathbb{N}}1_\E} \arrow[rr, "\simeq"] \arrow[dr, "\pi"'] & & (\E^S)^{\mathbb{N}} \arrow[dl, "\colim"] \\
    & \E^S &
   \end{tikzcd}.
  \end{center}
  We have shown in \cref{eq:coproduct} that under this equivalence $C \to \coprod_{\mathbb{N}}1_\E$ corresponds to $F: \mathbb{N} \to \E^S$. 
  Let $G_Y: \mathbb{N} \to \E^S$ be the functor that corresponds to $\coprod_{\mathbb{N}}Y \to \coprod_{\mathbb{N}}1_\E$ and notice $G_Y$ is just the functor 
  with constant value $Y$. The equivalence of categories gives us an equivalence of mapping spaces 
  $ \Map_{\E_{/ \coprod_{\mathbb{N}}1_\E}}( \coprod_{\mathbb{N}} Y, C) \simeq \Map_{\E^{\mathbb{N}}}(G_Y, F).$
  However in $\E^{\mathbb{N}}$ the mapping space is just the product of the individual mapping spaces, which means we get
  $\Map_\E(G_Y, F) \simeq \prod_{\mathbb{N}} \Map_\E(Y,F(n)). $
  Hence, $\E^S$ also has products. 
  Thus, by \cite[Proposition 4.4.2.6]{lurie2009htt} it has all countable limits and colimits.
  \par 
  Finally, we want to prove that if $\coprod_{\mathbb{N}} 1_\E$ exists then it is the natural number object, with maps $s: \coprod_{\mathbb{N}} 1_\E \to \coprod_{\mathbb{N}} 1_\E$ induced by successor map 
  $\mathbb{N} \to \mathbb{N}$ and the map $o: 1_\E \to \coprod_{\mathbb{N}} 1_\E$ induced by the inclusion $\{1 \} \hookrightarrow \mathbb{N}$.
  We will prove that it is a Freyd natural number object. 
  From the fact that $\mathbb{N} = \{1\} \coprod \mathbb{N} \backslash \{1\}$ it immediately follows that 
  \begin{center}
   \begin{tikzcd}[row sep=0.25in, column sep=0.25in]
    \emptyset \arrow[r] \arrow[d] & \coprod_{\mathbb{N}} 1_\E \arrow[d] \\
    1_\E \arrow[r] & \coprod_{\mathbb{N}} 1_\E
   \end{tikzcd}
  \end{center}
  is a pushout square. 
  Moreover, the diagram 
  \begin{center}
   \begin{tikzcd}[row sep=0.25in, column sep=0.5in]
     \coprod_{\mathbb{N}} 1_\E \arrow[r, shift left=0.05in, "s"] \arrow[r, shift right=0.05in, "id"']& \coprod_{\mathbb{N}} \arrow[r] & 1_\E
   \end{tikzcd}
  \end{center}
  is a coequalizer diagram. Indeed, the diagram on the left hand side is just the colimit of the poset $1_\E: (\mathbb{N}, \leq) \to \E$ constantly valued at $1_\E$ 
  and this poset is a contractible and thus the colimit is just $1_\E$.
 \end{proof}

 \subsection{Internal Infinite Coproducts and Sequential Colimits} \label{Subsec Internal Infinite Coproducts and Sequential Colimits}
 In general, an elementary $(\infty,1)$-topos does not have infinite colimits and limits.
 However, the existence of a natural number object in an elementary $(\infty,1)$-topos implies that we can construct certain 
 structures that behave like infinite colimits without being external colimits. 
 For this subsection let $\E$ be an elementary $(\infty,1)$-topos and fix a universe $\U$ in $\E$. 
  In order to simplify notation we will use following conventions:
  \begin{center}
   \begin{tabular}{|c|c|}
    \hline 
    Previous Section & This Section \\ \hline 
    $1_\E$ & $1$ \\ \hline 
    $\mathbb{N}_\E$ & $\mathbb{N}$ \\ \hline 
    $o: 1_\E \to \mathbb{N}_\E$ & $0: \mathbb{N}$ \\ \hline 
    $p: 1_\E \to \mathbb{N}_\E$ & $n:\mathbb{N}$ \\ \hline 
    $sp: 1_\E \to \mathbb{N}_\E$ & $n+1 :\mathbb{N}$ \\ \hline
   \end{tabular}
  \end{center}

  A sequence of objects $\{ A_n \}_{n: \mathbb{N}}$ is a map $\{ A_n \}_{n: \mathbb{N}} : \mathbb{N} \to \U$.
  For a given sequence of objects $\{ A_n \}_{n:\mathbb{N}}$, we define the {\it internal coproduct}, $\ds \sum_{n: \mathbb{N}} A_n$ as the pullback
  \begin{center}
   \pbsq{\ds \sum_{n: \mathbb{N}} A_n }{\U_*}{\mathbb{N}}{\U}{}{p_A}{}{\{ A_n \}_{n: \mathbb{N}}}
  \end{center}
  Let $n: \mathbb{N}$. Then we denote the fiber by $A_n$ and think of it as the ``$n$-th object" in the sequence.
  In particular the first fiber is $A_0$ and we have a sequence $A_0, A_{1}, A_{2}, ... , $
  which justifies calling a map $\mathbb{N} \to \U$ a sequence. 

 It's interesting to see examples of internal coproducts.
  Recall that by definition of a universe, a map $1 \to \U$ corresponds to an object in $\E$. 
  Let $1$ be the constant final sequence $A_n = 1$. This means that it is the map $\mathbb{N} \to \U$ that factors 
  through the constant map $1 \to \U$ that classifies the object $1$. 
  Now, we want to prove that the internal coproduct, $\sum_{n: \mathbb{N}} 1$, is simply equivalent to $\mathbb{N}$.
  We have the diagram 
  \begin{center}
   \begin{tikzcd}[row sep=0.25in, column sep=0.4in]
     \ds \sum_{n: \mathbb{N}} 1 \arrow [d, "p_{1}"', "\simeq"] \arrow[r] \arrow[dr, phantom, "\ulcorner", very near start] & 
     1 \arrow[r] \arrow[d] \arrow[dr, phantom, "\ulcorner", very near start] & \U_* \arrow[d] \\
     \mathbb{N} \arrow[r] & 1 \arrow[r] & \U
   \end{tikzcd}.
  \end{center}
  which gives us the equivalence $\sum_{n: \mathbb{N}} 1 \simeq \mathbb{N}$. 
  In fact we can easily generalize this.
  
  \begin{lemone} \label{lemma:coprod constant}
  	Let $X$ be an object in $\E$ and let $X: 1 \to \U$
  	be the map classifying $X$. Then the map $\mathbb{N} \to \U$ that factors through $1$ has pullback
  	$\mathbb{N} \times X$ which means $\sum_{n: \mathbb{N}} X \simeq X \times \mathbb{N}$.
  \end{lemone} 
 
 It is valuable to notice that the internal coproduct is homotopy invariant.
  Let $\{ A_n \}_{n:\mathbb{N}}$ and $\{B_n \}_{n: \mathbb{N}}$ be
  two sequences of objects such that $\{ A_n \}_{n:\mathbb{N}} \simeq \{B_n \}_{n: \mathbb{N}}$. 
  Then $\sum_{n: \mathbb{N}} A_n \simeq \sum_{n: \mathbb{N}} B_n$.
 This follows immediately from the homotopy invariance of the pullback and coequalizer.

 We now want to look at another class of infinite internal colimits, {\it sequential colimits}.
 A {\it sequential diagram} $\{ f_n: A_n \to A_{n+1} \}_{n: \mathbb{N}}$ is a sequence of objects 
  $\{ A_n \}_{n: \mathbb{N}}: \mathbb{N} \to \U$ as well as a choice of map
  \begin{center}
   \begin{tikzcd}[row sep=0.25in, column sep=0.25in]
    \ds \sum_{n: \mathbb{N}} A_n \arrow[rr, "\{ f_n \}_{n:\mathbb{N}}"] \arrow[dr] & & \ds \sum_{n: \mathbb{N}} A_{n+1} \arrow[dl] \\ 
     & \mathbb{N} & 
   \end{tikzcd}.
  \end{center} 
  For any $n:\mathbb{N}$ we get a map $f_n: A_n \to A_{n+1}$.
  Thus we will use following notation for a sequential diagram
  $$A_0 \xrightarrow{ \ \ f_0 \ \ } A_1 \xrightarrow{ \ \ f_1 \ \ } A_2 \xrightarrow{ \ \ f_2 \ \ }... \ . $$
 
  Let $\{ f_n \}_{n: \mathbb{N}}$ be a sequential diagram of the sequence of objects $A$. Then the {\it sequential colimit} of $f$ 
  is the coequalizer 
  \begin{center}
    \begin{tikzcd}[row sep=0.5in, column sep=0.5in]
    \ds \sum_{n: \mathbb{N}} A_n \arrow[r, shift left=0.05in, "f"] 
      \arrow[r, shift right=0.05in, "id_{ \sum_{n:\mathbb{N}} A_n }"'] & 
    \ds \sum_{n: \mathbb{N}} A_n \arrow[r]  & 
    A_{\infty}
   \end{tikzcd}.
  \end{center}
 
 Let us compute some examples.
   Let $X$ be an object in $\E$. Then we showed in \cref{lemma:coprod constant} that the constant coproduct on $X$ is $\mathbb{N} \times X$.
  The sequential diagram that correspond to the sequence 
  $$X \xrightarrow{ \ \ id_X \ \ } X \xrightarrow{ \ \ id_X \ \ } X \xrightarrow{ \ \ id_X \ \ } \cdots $$
  is simply the map $s \times id_X : \mathbb{N} \times X \to \mathbb{N} \times X$.
  Thus the sequential colimit of the constant sequence is the coequalizer of the diagram 
  \begin{center}
    \begin{tikzcd}[row sep=0.5in, column sep=0.5in]
    \mathbb{N} \times X  \arrow[r, shift left=0.05in, "s \times id_X"] 
      \arrow[r, shift right=0.05in, "id"'] & 
    \mathbb{N} \times X \arrow[r]  & 
    X_{\infty}
   \end{tikzcd}.
  \end{center}
  which, by the colimit condition of Freyd natural number objects \cref{Def Freyd NNO}, is simply $1 \times X \simeq X$.
 
 Notice we can recover infinite coproducts using sequential colimits. For that we first need the appropriate construction.
   
   \begin{propone}
    Let $X$ be an object in $\E$. The sequential colimit of the sequential diagram
    $$X \xrightarrow{ \ \ \iota_1 \ \ } X \coprod X \xrightarrow{ \ \ \iota_1 \ \ } (X \coprod X) \coprod X ... $$
    (defined below)
    is the infinite coproduct $\ds \sum_{n: \mathbb{N}} X = X \times \mathbb{N}$.
   \end{propone}
   
   \begin{proof}
   To simplify notation we will first prove the result for $X = 1$.
   First we will construct the sequence described in the proposition.
   Let $ - \coprod 1 : \U \to \U$   
   be the map of universes that corresponds to map $\E \to \E$ that takes an object 
   $Y$ to $Y \coprod 1$. This gives maps 
   $$ 1 \xrightarrow{ \ \ \emptyset \ \ } \U \xrightarrow{ \ \  - \coprod 1 \ \ } \U.$$
   By the property of the natural number object, we thus get a map $Fin_n: \mathbb{N} \to \U$, 
   which gives us a map $p_{Fin}: \sum_{n: \mathbb{N}} Fin_n \to \mathbb{N}$
   \par
   By its definition we have  $0^*(p_{Fin}) = \emptyset$ and $(n+1)^*(p_{Fin}) = n^*(p_{Fin}) \coprod 1$ for all $n: \mathbb{N}$. 
   This implies that $\sum_{n: \mathbb{N}} Fin_n \cong \mathbb{N}_1$ and 
   $p_{Fin} = s\pi_2$. In other words the infinite coproduct is the generic finite cardinal discussed before. 
   The map $\mathbb{N}_1 \to \mathbb{N}_1$ stated in the proposition corresponds to the map 
   $id \times s: \mathbb{N}_1 \to \mathbb{N}_1$, which is just the restriction of the
   map $id_{\mathbb{N}} \times s: \mathbb{N} \times \mathbb{N} \to \mathbb{N} \times \mathbb{N}$. 
   Indeed, the fact that the map $id \times s$ restricts just reflects the elementary fact that $n \dot - m = 0$ implies $n \dot - (m+1) = 0$.
   \par 
   Up to here we constructed a sequence of objects that we showed is equivalent to $\mathbb{N}_1$ and a 
   sequence of maps $id_{\mathbb{N}} \times s$. Intuitively it is the sequence of finite cardinals, where at each step
   we add one more element. We want to find the sequential colimit of this sequence. 
   We have following diagram
   \begin{center}
    \begin{tikzcd}[row sep=0.1in, column sep=0.25in]
     & \mathbb{N}_1 \arrow[r, shift left=0.05in, "id \times s"] 
       \arrow[r, shift right=0.05in, "id"'] \arrow[dd, hookrightarrow] & \mathbb{N}_1 \arrow[dd, hookrightarrow] \arrow[r] & 
     Fin_{\infty} \arrow[dd, hookrightarrow]\\
     \mathbb{N} \arrow[rrrd, "id" description, bend left=20] \arrow[dr, "o \times id"] \arrow[ur, "o \times id"] \\
     & \mathbb{N} \times \mathbb{N} \arrow[r, shift left=0.05in, "id \times s"] \arrow[r, shift right=0.05in, "id"'] 
     & \mathbb{N} \times \mathbb{N} \arrow[r] & \mathbb{N}
    \end{tikzcd}.
  \end{center}
   
   As $\mathbb{N}_1$ is a subobject of $\mathbb{N} \times \mathbb{N}$ we know that $Fin_{\infty}$ is a 
   subobject of $\mathbb{N}$. On the other hand it receives a map from the maximal subobject $\mathbb{N}$ and thus must be 
   the maximal subobject which implies that $Fin_{\infty} \cong \mathbb{N}$.
   \par 
   Thus we conclude that $\mathbb{N}$ is the sequential colimit of successive finite cardinals internal to $\E$.
   We can now generalize this result to any object $X$ by simply using the argument made before for the map 
   $$ 1 \xrightarrow{ \ \ \emptyset \ \ } \U \xrightarrow{ \ \  - \coprod X \ \ } \U . \qedhere$$
   \end{proof}
   
    There is an alternative way to construct the sequential diagram  
     $$X \xrightarrow{ \ \ \iota_1 \ \ } X \coprod X \xrightarrow{ \ \ \iota_1 \ \ } (X \coprod X) \coprod X \xrightarrow{ \ \ \iota_1 \ \ } \cdots \ .$$
    However, for that we have to first review some concepts.

     A complete Segal universe $\U_\bullet$ is a simplicial object in $\E$ that represents the target fibration 
     from the arrow category $target: Arr(\E) \to \E$.
     We have proven in \cite[Theorem 3.15]{rasekh2018elementarytopos} that every universe $\U$ can be extended to a complete Segal universe $\U_\bullet$.
     As before, a point $1 \to \U_0$ corresponds to an object $A$ in $\E$. Moreover, a point $1 \to \U_1$ corresponds to a morphism $f$.
     \par 
     As $\U_\bullet$ is a simplicial object it comes with a map $(s,t): \U_1 \to \U_0 \times \U_0$. It takes a morphism 
     $f: A \to B$ to the source and target 
     $(A,B)$. Thus in order to find the source and target of a morphism we simply apply $(s,t)$.
 
   Having reviewed complete Segal universes, we can use them to build the sequential diagram in a different way.
    We want to build a sequential diagram by constructing a map from the natural number object to $\U_1$.
    For that we construct a map $\U_1 \to \U_1$, which means we have to build an endofunctor of the arrow category. 
    Let $- \coprod id_X : \U_1 \to \U_1$ be the map that corresponds to the functor that takes a morphism $g:A \to B$ to the morphism 
    $g \coprod id_X: A \coprod X \to B \coprod X$. 
    \par 
    By initiality, the map $(\emptyset \to X): 1 \to \U_1$ and $[ - \coprod id_X]: \U_1 \to \U_1$ gives us a map 
    $ (\coprod_n X \to \coprod_{n+1} X) : \mathbb{N} \to \U_1$.
    As $\U_1$ classifies morphisms (by the previous remark) the map $p: \mathbb{N} \to \U_1$ classifies a commutative triangle
    \begin{center}
     \begin{tikzcd}[row sep=0.25in, column sep=0.25in]
       F_0 \arrow[dr] \arrow[rr, "f"] & & F_1 \arrow[dl] \\
       & \mathbb{N} &
     \end{tikzcd}
    \end{center}
    We have to determine $F_0$, $F_1$ and $f$. By definition, $F_0 \to \mathbb{N}$ and $F_1 \to \mathbb{N}$ are classified by
    $sp: \mathbb{N} \to \U_0$, $tp: \mathbb{N} \to \U_0$, where $(s,t): \U_1 \to \U_0 \times \U_0$ is the source-target map
    from the previous remark.
    \par 
    Thus by construction, $F_0 = \sum_{n:\mathbb{N}} (\coprod_n X)$ and $F_1 = \sum_{n:\mathbb{N}} (\coprod_{n+1} X)$.
    Moreover, the fiber of $f$ over $n: \mathbb{N}$ corresponds to the map $\iota_1: \coprod_n X \to \coprod_{n+1} X = (\coprod_n X) \coprod X$.
    Thus the map $\mathbb{N} \to \U_1$ gives us the desired sequential diagram
     $$X \xrightarrow{ \ \ \iota_1 \ \ } X \coprod X \xrightarrow{ \ \ \iota_1 \ \ } (X \coprod X) \coprod X \xrightarrow{ \ \ \iota_1 \ \ } \cdots \ .$$

   Similar to the coproduct case we also have a homotopy invariance and functoriality for sequential colimits.
    Let ($A_n$,$f_n$) and ($B_n$,$g_n$) be two sequential diagrams. A {\it natural transformation} between the diagrams is 
    a commutative diagram 
    \begin{center}
     \comsq{\sum_{n: \mathbb{N}} A_n}{\sum_{n: \mathbb{N}} A_n}{\sum_{n: \mathbb{N}} B_n}{\sum_{n: \mathbb{N}} A_n}{
     \{ f_n \}_{n: \mathbb{N}}}{F}{F}{\{ g_n \}_{n: \mathbb{N}}}.
    \end{center}
    We will usually denote a natural transformation as a map $F: \sum_{n: \mathbb{N}} A_n \to \sum_{n: \mathbb{N}} B_n$ in 
    order to simplify notation.
   
    Let ($A_n$,$f_n$) and ($B_n$,$g_n$) be two sequential diagrams. A natural transformation $F$
    induces a map of colimits $F_{\infty}: A_{\infty} \to B_\infty$. Moreover, if $F$ is an equivalence 
    then $F_{\infty}$ is an equivalence.
   Finally, we can also prove a cofinality result for sequential colimits. 
   
   \begin{theone}
    Let $\{ f_n: A_n \to A_{n+1} \}_{n: \mathbb{N}}$ be a sequential diagram. Then $\{f_n \}_{n: \mathbb{N}}$ 
    has the same sequential colimit as $\{f_{n+1} \}_{n: \mathbb{N}}$.
   \end{theone}
   
   \begin{proof}
    For the purpose of this proof we denote the sequential colimit of $f_{n+1}$ by $A_{\infty + 1}$.
    Recall that we have an isomorphism $(o,s): 1 \coprod \mathbb{N} \xrightarrow{ \ \cong \ } \mathbb{N}$. 
    Pulling it back gives us an isomorphism $A_0 \coprod \sum_{n: \mathbb{N}} A_{n+1} \xrightarrow{ \ \cong \ } \sum_{n: \mathbb{N}} A_{n}$. 
    This gives us a coequalizer diagram
    \begin{center}
     \begin{tikzcd}[row sep=0.5in, column sep=0.5in]
      \ds A_0 \coprod \sum_{n: \mathbb{N}} A_{n+1} \arrow[r, shift left=0.05in, "f_0 \coprod \{ f_{n+1} \}_{n:\mathbb{N}}"] 
      \arrow[r, shift right=0.05in, "id"'] & 
      \ds A_0 \coprod \sum_{n: \mathbb{N}} A_{n+1} \arrow[r]  & 
      A_{\infty}
     \end{tikzcd}.
   \end{center}
    However, we also know that the fiber of $\sum_{n: \mathbb{N}} A_{n+1}$ over $0$ is $A_{1}$
    which means 
    $$A_0 \coprod \sum_{n: \mathbb{N}} A_{n+1} \cong A_0 \coprod A_{1} \coprod_{A_{1}} \sum_{n: \mathbb{N}} A_{n+1}.$$
    We can thus rephrase the diagram in the following form.
    \begin{center}
     \begin{tikzcd}[row sep=0.5in, column sep=0.5in]
      \ds A_0 \coprod_{\emptyset} \sum_{n: \mathbb{N}} A_{n+1} \arrow[r, shift left=0.05in, "f_0 \coprod \{ f_{n+1} \}_{n:\mathbb{N}}"] 
      \arrow[r, shift right=0.05in, "id"'] & 
      \ds (A_0 \coprod A_{1}) \coprod_{A_{1}} \sum_{n: \mathbb{N}} A_{n+1} \arrow[r]  & 
      A_{\infty}
     \end{tikzcd}.
   \end{center}
   Thus the coequalizer diagram is a pushout of three coequalizer diagrams. Using the fact that colimit diagrams commute
   the diagram is thus equivalent to the pushout diagram 
   \begin{center}
    \begin{tikzcd}[row sep=0.25in, column sep=0.25in]
     A_1  \arrow[r] \arrow[d, "id_{A_1}"'] 
     & A_{\infty + 1} \arrow[d, "\simeq"] \\
     A_1 \arrow[r] 
     & A_{\infty} \arrow[ul, phantom, "\ulcorner", very near start]
    \end{tikzcd}.
   \end{center}
    This gives us the desired result that $A_{\infty + 1} \to A_\infty$ is an equivalence which finishes the proof.
   \end{proof}
   
   We can use the theorem $m$ times to get following corollary.
   
   \begin{corone}
    Let $\{ f_n \}_{n:\mathbb{N}}$ be a sequential diagram. Then $\{ f_n \}_{n:\mathbb{N}}$ has the same  sequential colimit as 
    $\{ f_{m+n} \}_{n:\mathbb{N}}$. 
   \end{corone}

 Using sequential colimits we can define infinite compositions.
  Let $\{ f_n:A_n \to A_{n+1} \}_{n:\mathbb{N}}$ be a sequential diagram. Then we get a natural transformation 
  $F: \sum_{n: \mathbb{N}} A_0 \to \sum_{n: \mathbb{N}} A_n$, where $ \sum_{n: \mathbb{N}} A_0 = A_0 \times \mathbb{N}$ is the constant sequence.
  This gives us a map of colimits $F_{\infty}: A_0 \to A_{\infty}$, as $A_0$ is the sequential colimit of the constant diagram $\mathbb{N} \times A_0$.
  Define the {\it infinite composition} $f_{\infty}: A_0 \to A_{\infty}$ as the map  $F_\infty$.

  The result of this subsection is that we can define various infinite colimits internally in an elementary $(\infty,1)$-topos. 
  These constructions can for example be used to study truncations in an elementary $(\infty,1)$-topos \cite{rasekh2018truncations}.
 
\bibliographystyle{alpha}
\bibliography{main}
\end{document}